\documentclass{article}
\usepackage[english]{babel}
\usepackage[T1]{fontenc}
\usepackage[utf8]{inputenc}
\usepackage{amsthm}
\usepackage{amsmath}
\usepackage{amssymb}
\usepackage{palatino}
\usepackage{graphicx}
\usepackage[colorlinks=true, citecolor =blue]{hyperref}
\usepackage{dsfont}
\usepackage[all]{xy}
\usepackage[top=4cm,bottom=4cm,left=3cm,right=3cm]{geometry}
\usepackage{enumitem}
\usepackage{mathtools}
\usepackage{subcaption}
\usepackage{authblk}

\newcommand{\R}{\mathds{R}}

\newcommand{\C}{\mathcal{C}}

\theoremstyle{plain}
\newtheorem{defi}{Definition}
\newtheorem{thm}{Theorem}
\newtheorem{lem}{Lemma}
\newtheorem{prop}{Proposition}

\newtheorem{rmk}{Remark}

\title{An elapsed time model for strongly coupled inhibitory and excitatory neural networks}
\author[a]{Nicolás Torres\thanks{Corresponding author. Email : nicolas.torres\_escorza@sorbonne-universite.fr}}
\author[a]{Benoît Perthame\thanks{Email : benoit.perthame@sorbonne-universite.fr}}
\author[b]{Delphine Salort\thanks{Email : delphine.salort@sorbonne-universite.fr}}
\affil[a]{Sorbonne Université, CNRS, Laboratoire Jacques-Louis Lions, UMR 7598. 4 Place Jussieu, 75005 Paris, France.}
\affil[b]{Sorbonne Université, CNRS, Laboratoire de Biologie Computationnelle et Quantitative, UMR 7238. 4 Place Jussieu, 75005 Paris, France.}

\date{August 2021}

\begin{document}

\title{A multiple time renewal equation for neural assemblies with elapsed time model}

\maketitle

\begin{abstract}
We introduce and study an extension of the classical elapsed time equation in the context of neuron populations that are described by the elapsed time since the last discharge, i.e., the refractory period. In this extension we incorporate the elapsed since the penultimate discharge and we obtain a more complex system of integro-differential equations. For this new system we prove convergence to stationary state by means of Doeblin’s theory in the case of weak non-linearities in an appropriate functional setting, inspired by the case of the classical elapsed time equation.
Moreover, we present some numerical simulations to observe how different firing rates can give different types of behaviors and to contrast them with theoretical results of both classical and extended models.
\end{abstract}

\vskip .7cm

\noindent{\makebox[1in]\hrulefill}\newline

\section{Introduction}
In the study and modelling of neural processes, population density models have proved to be a useful approach to understand brain phenomena at different scales. Among these models we mention for example the well-known integrate-and-fire model which describes the dynamics of the membrane potential and has been studied by several authors such as Carrillo et al. \cite{caceres2011analysis,carrillo2015qualitative}, Perthame et al. \cite{perthame2013voltage,perthame2017distributed,perthame2019derivation} and Zhou et al. \cite{liu2020rigorous} in different variants and approaches. Another class of population-based model is the elapsed time model, which has attracted the attention of many researchers. In this model we consider a neural network where neurons are described by their refractory period as the key variable, i.e. the elapsed time since the last discharge. After receiving some stimulation, neurons spike and interact with other neurons leading them to spike as well.

Like the integrate-and-fire model, the elapsed time model is closely related to the limit of stochastic processes at microscopic scale and the connection with Poisson processes was established in the work of Chevalier et al.~\cite{chevallier2015microscopic,chevallier2015mean}. Other important works on spiking neurons include Brunel~\cite{brunel2000dynamics}, Gerstner et al.~\cite{gerstner2002spiking},  Ly et al. in~\cite{ly2009spike} and Pham et al.~\cite{pham1998activity}. A recent survey is available by Schwalger et al.~\cite{schwalger2019mind}. Moreover, the elapsed time model has been studied from a mathematical and analytical point of view by several authors using different techniques such as Cañizo et al. in \cite{canizo2019asymptotic}, Kang et al. \cite{kang2015}, Mischler et al. in \cite{mischler2018,mischler2018weak} and the pioneer works of Pakdaman et al. in \cite{PPD,PPD2,PPD3}. The relation between integrate-and-fire and the elapsed time model was studied in Dumont et al. \cite{DH1, DH2}.

Different extensions of the elapsed time model have been considered by incorporating new variables such as spatial dependence and a connectivity kernel in Salort et al. \cite{torres2021elapsed} or a leaky memory variable in Fonte et al. \cite{fonte2021long}. The aim of the present work is to extend the classical elapsed time model by taking into account the elapsed time since the penultimate discharge in addition to the last one. In this context we study a multiple time renewal equation, which implies a more difficult analysis than that of the classical elapsed time equation. 

The extended model is described as follows. Let $n=n(t,s,a)$ the probability density of finding a neuron at time~$t$, such that the elapsed times since its last and penultimate discharge are respectively~$s$ and~$a$. For simplicity, we simply call~$s$ as the first elapsed time and~$a$ as the second one. Moreover, we assume that for all $t\ge 0$ the domain of definition of $n$ in the elapsed time variables is contained in the domain
$$\mathcal{D}\coloneqq\{(s,a)\in\R^2 \colon 0\le s\le a\}.$$

Neural dynamics are modelled through the following nonlinear renewal system 
\begin{equation}
\label{eqmain2t}
\left\{
\begin{matrix*}[l]
\partial_t n+\partial_s n+\partial_a n+p(s,a,X(t))n=0&t>0,\,a>s>0,\vspace{0.15cm}\\
n(t,s=0,a)=N(t,a)\coloneqq\int_0^\infty p(a,u,X(t))n(t,a,u)\,du&t>0,a>0,\vspace{0.15cm}\\
X(t)=\int_0^\infty N(t,a)\,da& t>0,\vspace{0.15cm}\\
n(t=0,s,a)=n_0(s,a)& a>s>0.
\end{matrix*}
\right.
\end{equation}
As in the classical elapsed time model the function $p\colon \mathcal{D}\times\R\to\R$ is the firing rate of neurons, which depends on the total activity $X(t)$. Furthermore, for the firing rate function $p$, we assume that there exist $\sigma,p_0,p_\infty>0$ such that
\begin{equation}
    \label{boundp2t}
    p_0\mathds{1}_{\{a>s\ge\sigma\}}\le p\le p_\infty.
\end{equation}
Thus, we get
\begin{equation}
    \label{boundX}
    0\le X(t)\le p_\infty,\qquad\forall t\ge0.
\end{equation} 
We assume for simplicity that $p\in W^{1,\infty}(\mathcal{D}\times\R)$, although most of the theoretical results are also valid for firing rates with simple jump discontinuities and the behavior of solutions does not depend  on this regularity assumption as we show in the numerical simulations. Furthermore, we say that the network is inhibitory if $p$ is decreasing with respect to the total activity~$X$ 
and excitatory if p is increasing. If in addition $\|\partial_X p\|_\infty$ is small, we say that System \eqref{eqmain2t} is under a weak interconnection regime.

The function $N(t,a)$ represents the flux discharging neurons conditioned to elapsed time since penultimate discharge, so that the total activity $X(t)$ corresponds to integrate with respect to all penultimate times. The boundary condition of $n$ at $s=0$ states that the second elapsed time resets to the first elapsed time. 

We assume that the initial data $n_0\in L^1(\mathcal{D})$ is a probability density so that System~\eqref{eqmain2t} formally verifies
\begin{equation}
\label{mass2t}
    \iint n(t,s,a)\,da\,ds = \iint n_0(s,a)\,da\,ds=1,\quad n(t,s,a)\ge 0\qquad\forall t\ge0.
\end{equation}
The multiple time renewal equation has been investigated in Fournier et al. \cite{fournier2021non} in the linear case, where a non-expanding distance was introduced via a coupling argument.

In the study of age-structured models, the entropy method has been a useful tool for proving convergence to the steady state. The main idea consists in finding a Lyapunov's functional $\mathcal{H}[n]$ and a dissipation functional $D_G[n]$ such that the solutions of the system satisfy
$$\dfrac{d}{dt}\mathcal{H}[n]=-D_\mathcal{H}[n]\le 0.$$
Thus if we can find a Poincaré inequality of the type $\lambda \mathcal{H}[n]\le D_\mathcal{H}[n]$ for some $\lambda>0$, we can deduce the exponential decay of $\mathcal{H}[n]$ by using the classical Gronwall's inequality, which eventually allows to deduce convergence to the steady state in some convenient norm. This method was developed in the works of \cite{michel2005general,perthame2006transport} with extensions to measure initial data in \cite{gwiazda2017generalized}, and it has been applied to different types of models. However, when such inequalities are not available the study of asymptotic behavior becomes more complex. 

Another important approach is Doeblin's theory, which was first introduced in the context of Markov chains \cite{doeblin1937proprietes} and later developed in the works of Harris \cite{harris1956existence}. This theory is an alternative to the classical entropy methods to prove convergence to the steady state for a wider class for firing rates. The main argument consists in proving that after a fixed time the solutions are uniformly bounded from below, implying the exponential convergence to equilibrium. We extend the ideas of Cañizo et al. \cite{canizo2019asymptotic} on the application of Doeblin's theory in the classical elapsed time model.

For a reference on Doeblin's theory, see for example Gabriel et al. \cite{gabriel2018measure}. A well-known extension of this theory is the Harris' theorem \cite{harris1956existence}, which has inspired several works such as Bansaye et al. \cite{bansaye2020ergodic}, Cañizo et al. \cite{canizo2020spectral} and Hairer \& Mattingly \cite{hairer2011yet}.
Moreover, convergence for the integrate-and-fire model has been proved in Perthame et al. \cite{perthame2019derivation} and Dumont et al. \cite{dumont2020mean} by means of Doeblin's theory.

Finally, we also remark that when $p$ does not depend on $a$, the probability density $m(t,s)\coloneqq\int_s^\infty n(t,s,a)\,da$ satisfies the equation
\begin{equation}
\label{eqm}
\left\{
\begin{matrix*}[l]
\partial_t m+\partial_s m+p(s,X(t))m=0&t>0,\,s>0,\vspace{0.15cm}\\
m(t,s=0)=X(t)=\int_0^\infty p(u,X(t))m(t,u)\,du&t>0,\vspace{0.15cm}\\
m(t=0,s)=\int_s^\infty n_0(s,a)\,da& s>0.
\end{matrix*}
\right.
\end{equation}
In other words, the probability with respect to the last elapsed time is a solution of the classical elapsed time equation. If in addition we consider a the firing rate of the form
$$p=\varphi(X(t))\mathds{1}_{\{s>\sigma\}},$$
with $\varphi\in W^{1,\infty}(\R)$ strictly positive and $\sigma>0$ a constant, we know from Caceres et al. \cite{torres2021elapsed} that the total activity $X(t)$ satisfies the integral equation
\begin{equation}
 \int_{t-\sigma}^{t} X(s)\,ds+\frac{X(t)}{\varphi(X(t))}=1,\qquad\forall t\ge\sigma.   
\end{equation}
Moreover, we know that the solutions of this integral equation may have different behaviors such as periodic solution and jump discontinuities. This gives us an idea of possible asymptotic behaviors that solutions of System~\eqref{eqmain2t} may exhibit.

The article is organized as follows. In Section \ref{wellp2t} we prove that System \eqref{eqmain2t} is well-posed in a suitable space for weak non-linearities. Starting with the asymptotic analysis for the linear case, we prove in Section \ref{convlinear2t} the existence of a stationary state and exponential convergence via Doeblin's theory. For the non-linear problem in the case of weak interconnections, we show in Section \ref{stationary2t} the uniqueness of the steady state and in Section \ref{convergence2t} we prove the exponential convergence via a perturbation argument. Finally in Section \ref{numerical2t} we present some examples of numerical simulations for different initial data and firing rates.

\section{Well-posedness for weak non-linearities}
\label{wellp2t}
We prove that System \eqref{eqmain2t} is well-posed under the weak interconnection regime. In order to do so, we start by studying an auxiliary linear problem where total activity is fixed and then we proceed
to prove well-posedness of system \eqref{eqmain2t} via a fixed point argument by  contraction.

\subsection{The linear problem}
Given $X\in\C_b[0,\infty)$, we consider the following linear problem
\begin{equation}
\label{eqlaux2t}
\left\{
\begin{matrix*}[l]
\partial_t n+\partial_s n+\partial_a n+p(s,a,X(t))n=0& t>0,a>s>0\vspace{0.15cm},\\
n(t,s=0,a)=N(t,a)\coloneqq\int_0^\infty p(a,u,X(t))n(t,a,u)\,du& t>0,a>0\vspace{0.15cm},\\
n(t=0,s,x)=n_0(s,x)\ge0& a>s>0.
\end{matrix*}
\right.
\end{equation}
We look for weak solutions satisfying $n\in\C_b([0,\infty),L^1(\mathcal{D}))$, so that $N\in\C_b([0,\infty), L^1(0,\infty))$ and $X\in\mathcal{C}_b[0,\infty)$.

\begin{lem}
\label{linearaux2t}
Assume that $n_0 \in L^1(\mathcal{D})$ is a probability density and $p\in W^{1,\infty}((0,\infty)\times\R)$ satisfies \eqref{boundp2t}. Then for a given $X\in\mathcal{C}_b[0,\infty)$, Equation \eqref{eqlaux2t} has a unique weak solution $n\in\C_b([0,\infty),L^1(\mathcal{D}))$ with $N\in\mathcal{C}_b([0,\infty), L^1(0,\infty))$ and $X\in\mathcal{C}_b[0,\infty)$. Moreover $n$ is non-negative and verifies the property~\eqref{mass2t}.
\end{lem}
In particular this lemma proves the property \eqref{mass2t} for the non-linear System \eqref{eqmain2t}.

\begin{proof}
	From the method of characteristics, we start by noticing that a solution of the linear System \eqref{eqlaux2t} satisfies the following fixed point equation
	\begin{equation}
	\label{char2times}
	\begin{split}
	n(t,s,x)=\Psi[n](t,s,x)&\coloneqq n_0(s-t,a-t)e^{-\int_0^t p(t'+s-t,t'+a-t,X(t'))dt'}\mathds{1}_{\{a>s>t\}}\\
    &\quad +N(t-s,a-s)e^{-\int_0^s p(s',s'+a-s,X(s'+t-s))ds'}\mathds{1}_{\{t,a>s\}},
	\end{split}
	\end{equation}
	with $N(t,a)=\int_0^\infty p(a,u,X(t))n(t,a,u)\,du$ depending on $n$.
	
	Let $T>0$ and $\mathcal{X}_T\coloneqq\{n\in\C_b([0,T],L^1(\mathcal{D}))\colon n(0)=n_0\}$, it readily follows that $\Psi$ maps $\mathcal{X}_T\to \mathcal{X}_T$. We prove by the Picard contraction theorem that $\Psi$ has a unique fixed point in $\mathcal{X}_T$ for $T>0$ small enough, i.e., there exists a unique weak solution of \eqref{eqlaux2t} defined on $[0,T]$. Consider $n_1,n_2\in \mathcal{X}_T$, we compute
	\begin{equation}
	\begin{split}
	\iint|\Psi[n_1]-\Psi[n_2]|(t,s,a)\,ds\,da&\le\int_0^t\int_s^\infty|N_1-N_2|(t-s,a-s)\,da\,ds\\
	&\le T\sup_{t\in[0,T]}\int_0^\infty|N_1-N_2|(t,a)\,da\\
	&\le T\,p_{\infty}\sup_{t\in[0,T]}\|n_1(t,s,a)-n_2(t,s,a)\|_{L^1(\mathcal{D})},
	\end{split}
	\end{equation}
	thus for $T<\frac{1}{p_{\infty}}$, we have proved that $\Psi$ is a contraction and there exists a unique $n\in X_{T}$ such that $\Psi[n]=n$. Since the choice of $T$ is independent of $n_0$, we can reiterate this argument to get a unique solution of \eqref{eqlaux2t}, which is defined for all $t\ge0$. 
	
From Formula \eqref{char2times} we can extend the notion of a weak solution for Equation \eqref{eqlaux2t} for an initial data  $n_0\in\left(\mathcal{M}(\mathcal{D}),\|\cdot\|_{M^1}\right)$, the space of finite regular measures on $\mathcal{D}$ with the norm of the total variation \eqref{Tvar}. Therefore we can redo the same argument to prove existence and uniqueness of a weak solution $n\in\C_b([0,\infty),\mathcal{M}(\mathcal{D}))$ with $N\in\mathcal{C}_b([0,\infty), \mathcal{M}(0,\infty))$ and $X\in\mathcal{C}_b[0,\infty)$.
	
	Next we prove the mass conservation property. For all $t\ge0$, consider $\mathcal{S}_t\colon \mathcal{M}(\mathcal{D})\to \mathcal{M}(\mathcal{D})$ the semi-group given by
	$$\mathcal{S}_t[f](s,a)=f(s-t,a-t)\mathds{1}_{\{a>s>t\}},$$
	whose infinitesimal generator is the operator $\mathcal{L}f=-\partial_s f-\partial_a f$. From Duhamel's formula, the solution of the fixed point problem \eqref{char2times} 
	also verifies the following equality
	\begin{equation}
	    n(t,s,a)= \mathcal{S}_t[n_0](s,a)+\int_0^t \mathcal{S}_{t-\tau}[\delta_{\{s=0\}}(s,a)N(\tau,a)]\,d\tau-\int_0^t \mathcal{S}_{t-\tau}[p(s,a,X(\tau))n(\tau,s,a)]\,d\tau,
	\end{equation}
	where $\delta_{\{s=0\}}(s,a)$ is the measure along the line $\{(0,a)\colon a\ge0\}$. This formula is translated as
	\begin{equation}
	\label{soln2t}
	\begin{split}
	    n(t,s,x)&=n_0(s-t,a-t)\mathds{1}_{\{a>s>t\}}+N(t-s,a-s)\mathds{1}_{\{t,a>s\}}\\
	    &\quad-\int_0^t p(s-t+\tau,a-t+\tau,X(\tau))n(\tau,s-t+\tau,a-t+\tau)\mathds{1}_{\{ a>s>t-\tau\}}\,d\tau,
	\end{split}
	\end{equation}
	and we get the mass conservation property by integrating with respect to $(s,a)$ on the domain $\mathcal{D}$. 
	
	Finally, since $n_0$ is non-negative then $\Psi$ preserves positivity, so by uniqueness of fixed point the corresponding solution $n$ must be non-negative.	
\end{proof}

\subsection{The non-linear problem}
We are now ready to prove that System \eqref{eqmain2t} is well-posed in the case of weak interconnection.
\begin{thm}[Well-posedness for weak interconnections]
	Assume that $n_0\in L^1(\mathcal{D})$ is a probability density and that $p\in W^{1,\infty}(\mathcal{D}\times\R)$ satisfies \eqref{boundp2t}. Then for 
	$$\|\partial_X p\|_\infty<1,$$ 
	System \eqref{eqmain2t} has a unique solution with $n\in\C_b([0,\infty),L^1(\mathcal{D})),\,N\in\C_b([0,\infty), L^1(0,\infty))$ and $X\in\mathcal{C}_b[0,\infty)$. Moreover $n$ verifies Condition \eqref{mass2t} for all $t>0$.
\end{thm}
\begin{proof}
	Consider $T>0$. We fix a function $X\in\C_b[0,\infty)$ and define the functions $n\in\C_b([0,\infty),L^1(\mathcal{D}))$ and $N\in\C_b([0,\infty), L^1(0,\infty))$ which are solutions of System \eqref{eqlaux2t} by Lemma \ref{linearaux2t}. Furthermore, the solution of this linear equation satisfies \eqref{mass2t}.
	
	So we have a solution of System \eqref{eqmain2t} defined on $[0,T]$ if $X$ satisfies for all $0\le t\le T$ and $x\in\Omega$, the following fixed point condition
	\begin{equation}
	\label{fixX}
		X(t)=\mathcal{T}[X](t)\coloneqq \int_0^\infty N[X](t,a)\,da.
	\end{equation}
	We prove that $\mathcal{T}$ defines for all $T>0$ an operator that maps $\mathcal{X}_T\to \mathcal{X}_T$ with $\mathcal{X}_T\coloneqq\C_b([0,T])$. First, we observe the following estimate
	\begin{equation}
	\label{bdN2t}
		\left|\int N(t,a)\,da\right|\le p_\infty,\quad\forall t\in [0,T],	
	\end{equation}
	and it is immediate that $\mathcal{T}[X]\in \mathcal{X}_T$. 
	
	We now prove that for $T$ small enough, $\mathcal{T}$ is a contraction. Let $X_1,X_2\in \mathcal{X}_T$ with their respective solutions $(n_1,N_1),(n_2,N_2)$ of System \eqref{eqlaux2t}. For the difference between $N_1$ and $N_2$ we have
	\begin{equation}
	\label{diffN2t}
		\begin{split}
		\int|N_1-N_2|(t,a)\,da &\le\iint|p(a,u,X_1)\,n_1(t,a,u)-p(a,u,X_2)\,n_2(t,a,u)|\,du\,da\\
		&\le \iint|p(a,u,X_1)-p(a,u,X_2)|\,n_1\,du\,da
		+\iint p(a,u,X_2)|n_1-n_2|(t,a,u)\,du\,da\\
		&\le \|\partial_X p\|_\infty\,\|X_1-X_2\|_\infty+p_{\infty}\|n_1-n_2\|_{L^1(\mathcal{D})}.
		\end{split}
	\end{equation}
	Now we have to estimate the difference between $n_1$ and $n_2$. From \eqref{soln2t} and estimate \eqref{diffN2t}, we get
	\begin{equation*}
	\|n_1-n_2\|_{L^1(\mathcal{D})}\le 2T\,\|\partial_X p\|_\infty\,\|X_1-X_2\|_\infty+2Tp_{\infty}\|n_1-n_2\|_{L^1(\mathcal{D})}.
	\end{equation*}
	Then, for $T<\tfrac{1}{2p_{\infty}}$ we obtain
	\begin{equation}
	\|n_1-n_2\|_{L^1(\mathcal{D})}\le\frac{2T\|\partial_X p\|_\infty}{1-2Tp_{\infty}}\|X_1-X_2\|_{\infty}.
	\end{equation}
	Finally by using again estimate \eqref{diffN2t}, the operator $\mathcal{T}$ satisfies
	\begin{equation}
	\|\mathcal{T}[X_1]-\mathcal{T}[X_2]\|_\infty\le \|\partial_X p\|_\infty\left(1+\frac{2Tp_\infty}{1-2Tp_\infty}\right)\|X_1-X_2\|_{\infty}
	\end{equation}
	Hence for $\|\partial_X p\|_\infty <1$ and $T$ small enough, $\mathcal{T}$ is a contraction.

	From Picard's fixed point we get a unique $X\in \mathcal{X}_T$ such that $\mathcal{T}[X]=X$, and this implies the existence of a unique solution of \eqref{eqmain2t} defined on $[0,T]$.
	Since estimate \eqref{bdN2t} is uniform in $T$, we can iterate this argument to get a unique solution of \eqref{eqmain2t} defined for all $t>0$.
	
	Furthermore, we conclude from this construction that the non-linear System \eqref{eqmain2t} satisfies \eqref{mass2t} like the linear System \eqref{eqlaux2t}.
\end{proof}

\section{Asymptotic behavior for the linear case}
\label{convlinear2t}
In order to study the behavior of System \eqref{eqmain2t}, we start by studying the case when $X\ge0$ is a fixed constant. Thus we consider the linear problem given by
\begin{equation}
\label{eql2t}
\left\{
\begin{matrix*}[l]
\partial_t n+\partial_s n+\partial_a n+p(s,a,X)n=0&t>0,\,a>s>0,\vspace{0.15cm}\\
n(t,s=0,a)=N(t,a)\coloneqq\int_0^\infty p(a,u,X)n(t,a,u)\,du&t>0,a>0,\vspace{0.15cm}\\
n(t=0,s,a)=n_0(s,a)& a>s>0.
\end{matrix*}
\right.
\end{equation}

To determine the behavior of System \eqref{eql2t}, we consider $(n_X,N_X)$ as the solution of the steady state problem given by

\begin{equation}
\label{eql2te}
\left\{
\begin{matrix*}[l]
\partial_s n+\partial_a n+p(s,a,X)n=0&a>s>0,\vspace{0.15cm}\\
n(s=0,a)=N(a)\coloneqq\int_0^\infty p(a,u,X)n(a,u)\,du&a>0,\vspace{0.15cm}\\
\end{matrix*}
\right.
\end{equation}

In the classical elapsed time model the generalized relative entropy inequality is a well-known property of this class of age-structured models. In the same way, we can prove this property for the linear System \eqref{eql2t}.

\begin{prop}[Generalized relative entropy]
\label{GRE2t}
Assume there exists a steady solution of the linear System \eqref{eql2t} with $n_X,N_X>0$. Then for all convex functions $H\colon[0,\infty)\to[0,\infty)$ with $H(0)=0$, the solution $n$ of the linear System \eqref{eql2t} satisfies
\begin{equation}
\label{entropy2t}
       \begin{split}
        \frac{d}{dt}\iint n_X(s,a) H\left(\frac{n(t,s,a)}{n_X(s,a)}\right)\,da\,ds =-D_H[n(t,s)]\le 0\qquad\forall t\ge 0,\\
        D_H[n(t,s,a)]=\iint p(s,a,X)H\left(\frac{n(t,s,a)}{n_X(s,a)}\right)\,da\,ds-\int N_X(a)H\left(\frac{N(t,a)}{N_X(a)} \right)\,da,
    \end{split} 
\end{equation}
and in particular the steady state is unique.
\end{prop}

\begin{proof}
In order to prove the relative entropy property, we follow the arguments in \cite{caceres2011analysis}. We start by noticing the following identities
\begin{equation}
\label{identities}
\partial_s n=n_X \partial_s\left(\frac{n}{n_X}\right)+\frac{n}{n_X}\partial_s n_X,\qquad\partial_a n=n_X \partial_a\left(\frac{n}{n_X}\right)+\frac{n}{n_X}\partial_a n_X,
\end{equation}
and for simplicity we reformulate Equation \eqref{eql2t} as follows

\begin{equation}
\label{reformul1}
\left\{
\begin{matrix*}[l]
\partial_t n+\partial_s n+\partial_a n+p(s,a,X)n=\delta_{\{s=0\}}(s,a)N(t,a)&t>0,\,a>s>0,\vspace{0.15cm}\\
n(t,s=0,a)=0&t>0,a>0,\vspace{0.15cm}\\
n(t=0,s,a)=n_0(s,a)& a>s>0,
\end{matrix*}
\right.
\end{equation}
where $\delta_{\{s=0\}}(s,a)$ is the measure along the line $\{(0,a)\colon a\ge0\}$. In the same way, we reformulate the corresponding steady state problem \eqref{eql2te}.
\begin{equation}
\label{reformul2}
\left\{
\begin{matrix*}[l]
\partial_s n_X+\partial_a n_X+p(s,a,X)n_X=\delta_{\{s=0\}}(s,a)N_X(a)&t>0,\,a>s>0,\vspace{0.15cm}\\
n_X(t,s=0,a)=0&a>0.\vspace{0.15cm}\\
\end{matrix*}
\right.
\end{equation}
Hence by using the identities \eqref{identities} along with Equations \eqref{reformul1} and \eqref{reformul2}, we get the following equation for $\frac{n}{n_X}$
\begin{equation*}
  \partial_t\left(\frac{n}{n_X}\right)+ \partial_s\left(\frac{n}{n_X}\right)+ \partial_a\left(\frac{n}{n_X}\right)=\delta_{\{s=0\}}(s,a)\frac{N_X}{n_X}\left(\frac{N}{N_X}-\frac{n}{n_X}\right)
\end{equation*}
and if we multiply this equality by $H'\left(\frac{n}{n_X}\right)$, we get
\begin{equation*}
  \partial_t H\left(\frac{n}{n_X}\right)+ \partial_s H\left(\frac{n}{n_X}\right)+ \partial_a H\left(\frac{n}{n_X}\right)=\delta_{\{s=0\}}(s,a)\frac{N_X}{n_X}\left(\frac{N}{N_X}-\frac{n}{n_X}\right)H'\left(\frac{n}{n_X}\right).
\end{equation*}
Therefore, by multiplying the latter equality by $n_X$ and using Equation \eqref{reformul2}, we have the corresponding equation for $u=n_X H\left(\tfrac{n}{n_X}\right)$
\begin{equation}
\label{eqHn}
  \partial_t u+\partial_s u+\partial_a u+p(s,a,X)u=\delta_{\{s=0\}}(s,a)N_X\left[\left(\frac{N}{N_X}-\frac{n}{n_X}\right)H'\left(\frac{n}{n_X}\right)+ H\left(\frac{n}{n_X}\right)\right].
\end{equation}
Finally, by noticing the following limit
$$\lim_{s\to0}\frac{n(t,s,a)}{n_X(s,a)}=\frac{N(t,a)}{N_X(a)},\qquad\textrm{for a.e.}\quad t,a>0,$$
we conclude the generalized relative entropy property \eqref{entropy2t} by integrating Equation \eqref{eqHn} with respect to $(s,a)$ on the domain $\mathcal{D}$. Moreover, we observe that $D_H[\cdot]$ is non-negative by applying Jensen's inequality with the probability measure $d\mu=p(a,y)\frac{n_X(a,y)}{N_X(a)}\,dy$ for each $a>0$. In particular when $H$ is strictly convex and $D_H[n]=0$, we deduce that $\frac{n}{n_X}$ is constant and subsequently we get $n=n_X$, since both $n,n_X$ are probability densities. Therefore, the steady state is unique.
\end{proof}

If we consider the entropy method to prove exponential convergence for the linear Equation \eqref{eql2t} in $L^1(\mathcal{D})$, we have following equality for $H(\cdot)=|\cdot|$
$$\frac{d}{dt}\iint|n-n_X|\,da\,ds =\int\left|\int p(n-n_X)\,da\right|ds-\iint p|n-n_X|\,da\,ds\le 0,$$
and the $L^1$ Poincaré inequality for the right-hand side is not available since the condition\\ $\int_s^\infty(n-n_X)\,da=0$ is not fulfilled. 

Furthermore, in Theorem \eqref{GRE2t} we assumed that $n_X$ and $N_X$ are strictly positive, which is not necessarily true. Unlike the classical elapsed time model, there exist solutions where $n_X$ and $N_X$ vanish for some values of $(s,a)$. Indeed, consider for example $p(s,a,X)=\mathds{1}_{\{s>1\}}$ which satisfies the bounds \eqref{boundp2t} and observe that $N_X$ satisfies Equation \eqref{Idcompact}, implying that $N_X(a)$ vanishes for $a<1$ and subsequently we see from Formula \eqref{solneq2t} that $n_X$ vanishes when $a-s<1$.

Due to the limitations of the entropy method approach we will make use of Doeblin's theory, which will be the key ingredient in proving convergence to steady state. In this context we start by reminding the useful concepts in order to apply Doeblin's theorem. Consider $(\mathcal{M}(\mathcal{X}),\|\cdot\|_{M^1})$ the space of finite signed measures with the norm of the total variation
\begin{equation}
\label{Tvar}   
    \|\mu\|_{M^1}\coloneqq \int_X \mu_+\int_X \mu_-,
\end{equation}
where $\mu=\mu_+-\mu_-$ is the Hahn–Jordan decomposition of the measure $\mu$ into its positive and negative parts. For simplicity of the computations, we will treat measures as if they were $L^1$ functions and we simply write the $L^1$-norm instead of $M^1$-norm.

We now recall the definition of a Markov semigroup and Doeblin's condition.
\begin{defi}[Markov semi-group]
	Let $(\mathcal{X},\mathcal{A})$ be a measure space and $P_t\colon \mathcal{M}(\mathcal{X})\to \mathcal{M}(\mathcal{X})$ be a linear semi-group. We say that $P_t$ is a Markov semi-group if $P_t \mu\ge0$ for all $\mu\ge0$ and $\int_X P_t \mu=\int_X \mu$ for all $\mu\in \mathcal{M}(\mathcal{X})$. In other words, $(P_t)$ preserves the subset of probability measures $\mathcal{P}(\mathcal{X})$.
\end{defi}	

\begin{defi}[Doeblin's condition]
	Let $P_t\colon \mathcal{M}(\mathcal{X})\to \mathcal{M}(\mathcal{X})$ be a Markov semi-group. We say that $(P_t)$ satisfies Doeblin's condition if there exist $t_0>0,\,\alpha\in(0,1)$ and $\nu\in \mathcal{P}(\mathcal{X})$ such that
	$$P_{t_0}\mu\ge\alpha\nu\quad\forall \mu\in\mathcal{P}(\mathcal{X}).$$
\end{defi}

Under this functional setting, we are now ready to state Doeblin's theorem as follows.

\begin{thm}[Doeblin's Theorem] 
	Let $P_t\colon \mathcal{M}(\mathcal{X})\to \mathcal{M}(\mathcal{X})$ be a Markov semi-group that satisfies Doeblin's condition. Then the semigroup has a unique equilibrium $\mu^*\in\mathcal{P}(\mathcal{X})$. Moreover, for all $\mu\in \mathcal{M}(\mathcal{X})$ we have
	$$\|P_t\mu-\langle \mu\rangle \mu^*\|_{M^1}\le\frac{1}{1-\alpha}e^{-\lambda t}\|\mu-\langle \mu\rangle \mu^*\|_{M^1}\quad\forall t\ge 0,$$
	with $\langle \mu\rangle=\int_X \mu$ and $\lambda=-\frac{\ln(1-\alpha)}{t_0}>0$.
\end{thm}
For a proof of Doeblin's Theorem, see for example \cite{gabriel2018measure}.

From Lemma \ref{eqlaux2t}, the solution of the linear problem \eqref{eql2t} determines a Markov semi-group acting on $L^1(\mathcal{D})$. By means of Doeblin's theory, the solutions of linear Equation \eqref{eql2t} converge exponentially to a unique steady state, as we assert in the following theorem. 

\begin{thm}
\label{conveq2t}
    Let $n_0\in L^1(\mathcal{D})$ be a probability density and assume that $p$ smooth satisfies Assumption~\eqref{boundp2t}. Then for a fixed $X>0$, there exists a unique stationary solution $n_X(s,a)\in L^1(\mathcal{D})$ of the linear Equation \eqref{eql2t} satisfying $\iint n_X(s,a)\,da\,ds=1$. Moreover, the corresponding solution of Equation~\eqref{eql2t} satisfies
	$$\|n(t)-n_X\|_{L^1_{s,a}}\le\frac{1}{1-\alpha}e^{-\lambda t}\|n_0-n_X\|_{L^1_{s,a}}\qquad\forall t\ge 0,$$
	with $\alpha=\frac{1}{2}p_0^2\sigma^2 e^{-3p_{\infty}\sigma}$ and $\lambda=-\frac{\log(1-\alpha)}{3\sigma}>0$.
\end{thm}

In order to obtain the result, we show that after some time the solution of the linear problem is uniformly bounded from below for all probability densities. Thus from Doeblin's theorem we get the exponential convergence to equilibrium.

\begin{lem}
\label{doeblinlb}
Assume \eqref{mass2t} and \eqref{boundp2t}. Let $n(t,s,a)$ be a solution of \eqref{eql2t}, then there exist $t_0>0,\,\alpha\in(0,1)$ and a probability density $\nu\in L^1$ such that
$$n(t_0,s,a)\ge \alpha \nu(s,a).$$
\end{lem}

\begin{proof}
The main idea of the proof is to control the mass transported along the lines of direction $(1,1)$. Firstly, we observe the transport of the initial data $n_0$. From Assumption \eqref{boundp2t} and the characteristics Formula \eqref{char2times} the following inequality holds
\begin{equation}
\label{lowermass}
    \int_t^\infty\int_s^\infty n(t,s,a)\,da\,ds\ge e^{-p_\infty t},\qquad\forall t\ge\sigma.
\end{equation}

Secondly, we see the mass that returns at $s=0$. From \eqref{lowermass} we get for all $t\ge\sigma$
\begin{equation} 
\int_t^\infty n(t,s=0,a)\,da=\int_t^\infty N(t,a)\,da\ge p_0\int_t^\infty\int_a^\infty n(t,a,u)\,du\,da\ge p_0e^{-p_\infty t}. 
\end{equation}
This means that we reduced by one dimension the problem of finding the uniform lower bound. For $t\ge\sigma$ the mass of the region $\{(s,a)\colon a>s>t\}$ concentrates in the line $\{(0,a)\colon a\ge t\}$, as we see in Figure \ref{transport1}.

\begin{figure}[ht!]
    \centering
    \includegraphics[scale=0.75]{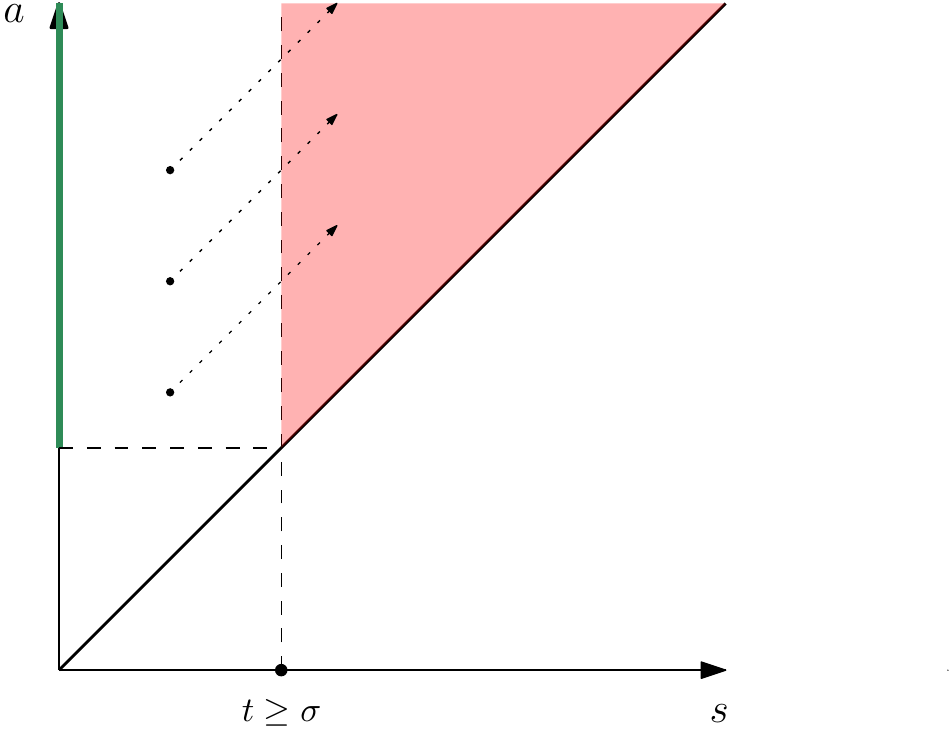}
    \caption{First reduction of dimension. For a $t\ge\sigma$, all points in $\mathcal{D}$ are transported to the red region, which has a total mass of at least $e^{-p_\infty t}$. Then a mass of at least $p_0e^{-p_\infty t}$ returns to the green line.}
    \label{transport1}
\end{figure}

Thirdly, in order to control the point values of $n(t,s,a)$, we regard the values of $N(t,a)$. Observe that from Formula~\eqref{char2times} we have
\begin{equation}
\label{nlowerb}
\begin{split}
   n(t,s,a)&\ge N(t-s,a-s)e^{-\int_0^s p(s',s'+a-s,X)ds'}\mathds{1}_{\{t,a>s\}}\\
   &\ge N(t-s,a-s)e^{-p_\infty s}\mathds{1}_{\{t,a>s\}},
\end{split}
\end{equation}
thus for $a\ge\sigma$ and $t-a>\sigma$, we obtain by using again Assumption \eqref{boundp2t} that
\begin{equation}
    \begin{split}
        N(t,a)&\ge p_0\int_a^\infty n(t,a,u)\,du\\
        &\ge p_0 e^{-p_\infty a}\int_a^\infty N(t-a,u-a)\,du=p_0 e^{-p_\infty a}\int_0^\infty N(t-a,u)\,du\\
        &\ge p_0 e^{-p_\infty a}\int_{t-a}^\infty N(t-a,u)\,du\ge  p_0^2 e^{-p_\infty t}.
    \end{split}
\end{equation}
This means we reduced the problem of finding the uniform lower bound by one dimension again, as we see in Figure~\ref{transport2}.

\begin{figure}[ht!]
    \centering
    \includegraphics[scale=0.75]{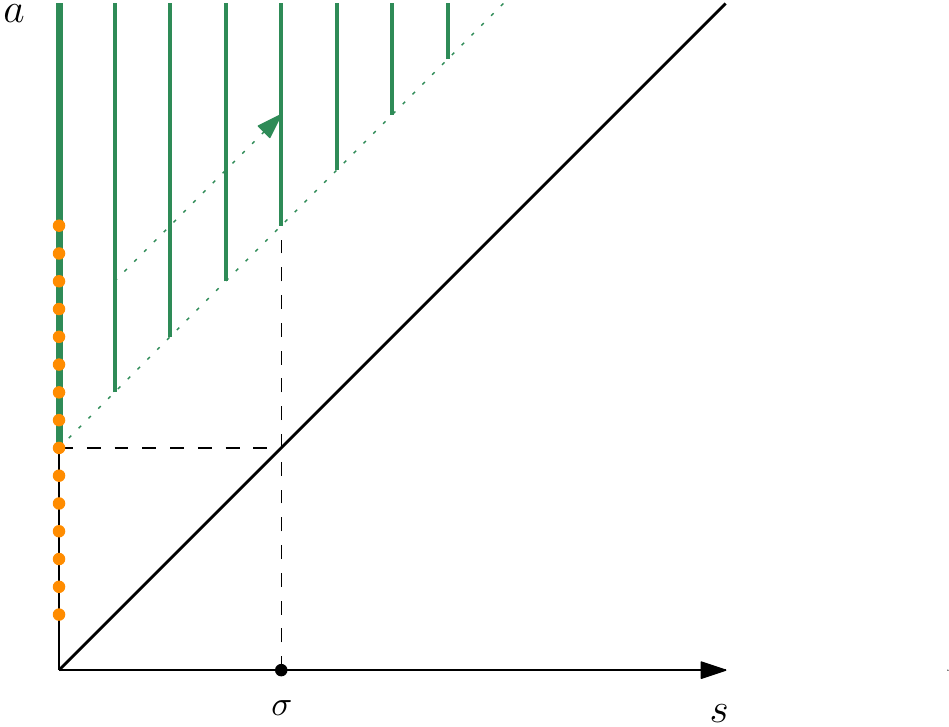}
    \caption{Second reduction of dimension. For $t\in[\sigma,2\sigma]$ the green lines are transported to the region where $s\ge\sigma$ and their mass is of at least $p_0e^{-2 p_\infty\sigma}$. Then the mass of each green line is concentrated in the orange points, whose values are at least $p_0^2e^{-2p_\infty\sigma}$.}
    \label{transport2}
\end{figure}

Finally, once we have estimated $N(t,a)$ from below, we come back to estimate \eqref{nlowerb} to conclude that for $a-s\ge\sigma$ and $t-a>\sigma$ we have
\begin{equation}
    \begin{split}
        n(t,a,s)&\ge N(t-s,a-s)e^{-p_\infty s}\mathds{1}_{\{ t,a>s\}}\\
        &\ge p_0^2 e^{-p_\infty t}\mathds{1}_{\{t-a,a-s>\sigma\}},
    \end{split}
\end{equation}
so that we can choose $t=3\sigma$ and conclude that 
$$n(3\sigma,a,s)\ge p_0^2 e^{-3p_\infty\sigma}\mathds{1}_{\{2\sigma>a>s+\sigma\}}.$$
Therefore we get the desired result with $t_0=3\sigma,\,\alpha=\frac{1}{2}\sigma^2p_0^2e^{-3p_\infty\sigma}\in(0,1)$ and $\nu$ given by
$$\nu(s,a)=\frac{2}{\sigma^2}\mathds{1}_{\{2\sigma>a>s+\sigma\}},$$
whose support is contained in orange region of Figure \ref{transport3}.

\begin{figure}[ht!]
    \centering
    \includegraphics[scale=0.75]{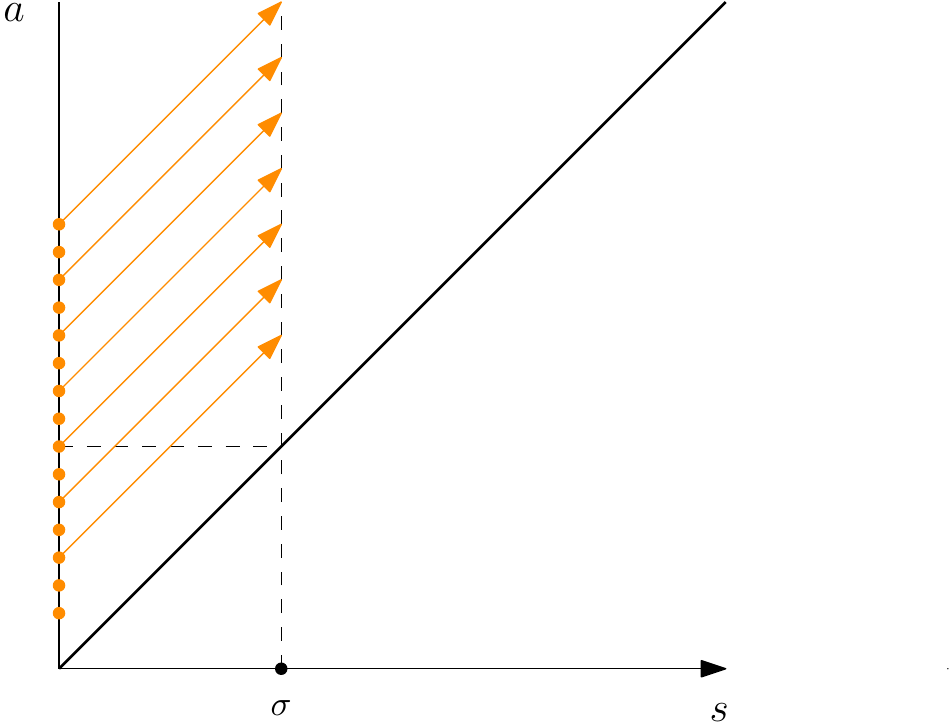}
    \caption{Finally for $t\in[2\sigma,3\sigma]$ the orange dots are transported to region where $s\ge\sigma$, which allows to construct a minorization function for Doeblin's Theorem.}
    \label{transport3}
\end{figure}
\end{proof}

From Lemma \ref{doeblinlb} the hypothesis of Doeblin's theorem are verified and Theorem \ref{conveq2t} readily follows.

Concerning the conditioned activity $N$ in System \eqref{eql2te}, we conclude from Theorem \ref{conveq2t} that for $X$ fixed, there is a unique stationary $N_X\in L^1(0,\infty)$ determined by the method of characteristics through the formula
\begin{equation}
\label{solneq2t}
    n_X(s,a)=N_X(a-s)\exp\left(-\int_0^s p(s',a-s+s',X)\,ds'\right),\qquad a>s.
\end{equation}
Replacing this expression in the boundary condition at $s=0$, we obtain the following integral equation for $N_X(a)$
\begin{equation}
    \label{Idcompact}
    N_X(a)=\mathcal{T}_X[N_X](a),
\end{equation}
with $\mathcal{T}_X\colon L^1(0,\infty)\to L^1(0,\infty)$ given by
\begin{equation*}
\begin{split}
\mathcal{T}_X[N](a)&\coloneqq \int_0^\infty p(a,u+a,X)\exp\left(-\int_0^a p(s',u+s',X)\,ds\right)N(u)\,du\\
    &=-\frac{\partial}{\partial a}\int_0^\infty \exp\left(-\int_0^a p(s',u+s',X)\,ds'\right)N(u)\,du.
\end{split}
\end{equation*}

Moreover, by integrating Equation \eqref{solneq2t} we get
\begin{equation}
\label{normal2t}
    \int_0^\infty\int_0^\infty N_X(a)\exp\left(-\int_0^s p(s',a+s',X)\,ds'\right)\,da\,ds=1.
\end{equation}
Therefore we conclude that finding a function $N\in L^1(0,\infty)$ satisfying Equation \eqref{Idcompact} and Condition \eqref{normal2t} is equivalent to finding a steady state $n_X(s,a)$ in Equation \eqref{eql2te}. The integral Equation \eqref{Idcompact} will play an important role in the analysis of the non-linear System \eqref{eqmain2t}, thus we prove the following two lemmas on the operator $\mathcal{T}_X$ that will be useful in the sequel.

\begin{lem}
\label{kerran}
Assume that $p$ Lipschitz satisfies Assumption \eqref{boundp2t}. For each $X>0$ the operator $\mathcal{T}_X$ is compact and it satisfies that $\dim\ker(I-\mathcal{T}_X)=1$, which is generated by a non-negative function, and $$\mathop{\textrm{ran}}(I-\mathcal{T}_X)=\left\{f\in L^1(0,\infty)\colon\int_0^\infty f(x)\,dx=0\right\}.$$
\end{lem}

\begin{proof}
The first step is to prove that $\mathcal{T}_X$ is a compact operator. This means we have to prove that the set $A=\{\mathcal{T}_X[f]\colon ||f||_1\le1\}$ is relatively compact in $L^1$.

First observe that $\|\mathcal{T}_X[f]\|_1\le p_\infty$ for all $f$ with $||f||_1\le1$, so $A$ is bounded.

Second, we prove that 
$$\int_r^\infty|\mathcal{T}_X[f](a)|da\to 0\qquad\textrm{uniformly when}\:r\to\infty.$$
Indeed for $r>\sigma$ we have
\begin{equation*}
\begin{split}
    \int_r^\infty|\mathcal{T}_X[f](a)|da&\le p_\infty\int_r^\infty\int_0^\infty|f(u)|e^{-\int_0^a p(s',u+s',X)\,ds'}\,du\,da\\
    &\le p_\infty\int_r^\infty\int_\sigma^\infty|f(u)|e^{-p_0(a-\sigma)}\,du\,da\\
    &\le p_\infty e^{p_0\sigma}\int_\sigma^\infty|f(u)|du\int_r^\infty e^{-p_0a}da\\
    &\le  p_\infty e^{p_0\sigma}\frac{e^{-p_0r}}{p_0}\to 0.
\end{split}
\end{equation*}

Now we prove the equicontinuity property. Observe that
\begin{equation*}
\begin{split}
    \frac{d}{da}\mathcal{T}_X[f](a)&=\int_0^\infty (\partial_s p+\partial_a p)(a,u+a,X)e^{-\int_0^a p(s',u+s',X)\,ds'}f(u)du\\
    &\quad-\int_0^\infty p(a,u+a,X)^2e^{-\int_0^a p(s',u+s',X)\,ds'}f(u)du,
\end{split}
\end{equation*}
thus for $f$ with $\|f\|_1\le 1$ we have
\begin{equation*}
    \int_0^\infty\left|\frac{d}{da}\mathcal{T}_X[f](a)\right|da\le \|\nabla p\|_\infty+p_\infty^2.
\end{equation*}
Therefore by the Kolmogorov-Frechet theorem we conclude that $A$ is relatively compact so the operator $\mathcal{T}_X$ is.

Furthermore, since $n_X$ is the unique steady state of Equation \eqref{eql2t} that is a probability density, from the linearity we deduce that any other function in $\ker(I-\mathcal{T}_X)$ is a multiple of $N_X$ and thus $\dim\ker(I-\mathcal{T}_X)=1$.

Next, we proceed to determine $\mathop{\textrm{ran}}(I-\mathcal{T}_X)$. Observe that adjoint operator $\mathcal{T}_X^*\colon L^\infty\to L^\infty$ is given by
$$\mathcal{T}_X^*[g](a)=\int_0^\infty p(u,u+a,X)\exp\left(-\int_0^u p(s',a+s',X)\,ds'\right)g(u)\,du,$$
and from Fredholm's alternative we get $\dim\ker(I-\mathcal{T}_X^*)=\dim\ker(I-\mathcal{T}_X)=1$. Since $\mathcal{T}_X^*[g]\equiv0$ for any constant function, we deduce that $\ker(I-\mathcal{T}_X^*)$ is the subspace of constant functions. Finally from orthogonality conditions we conclude that 
$$\mathop{\textrm{ran}}(I-\mathcal{T}_X)=\left\{f\in L^1(0,\infty)\colon\int_0^\infty f(x)\,dx=0\right\}.$$

\end{proof}

A direct consequence of Lemma \ref{kerran} is the following result
\begin{lem}
    Assume that $p$ is smooth respect to variable $X$, then $N_X(a)$ is also smooth with respect to $X$.
\end{lem}

\begin{proof}
Define the $F\colon L^1(0,\infty)\times (0,\infty)\to\mathop{\textrm{ran}} (I-\mathcal{T}_X)\times\R$ given by
$$F(N,X)=\begin{pmatrix}(I-\mathcal{T}_X)[N]\:,&\iint N(a) e^{-\int_0^a p(s',u+s',X)\,ds'}\,da\,ds-1\end{pmatrix},$$
so that for each $X$ we have $F(N_X(a),X)=0$. Observe that $D_NF$ is given by
$$D_NF[h]=\begin{pmatrix}(I-\mathcal{T}_X)[h]\:,&\iint h(a) e^{-\int_0^a p(s',u+s',X)\,ds'}\,da\,ds\end{pmatrix}$$
Thus by Lemma \ref{kerran} this operator is an isomorphism and from the implicit function theorem we conclude that $N_X(a)$ depends smoothly on $X$.
\end{proof}

\begin{rmk}
The lower bound condition \eqref{boundp2t} on the firing rate $p$ is important to verify the existence of a steady state for System \eqref{eqmain2t} and Doeblin's condition. For example, when we consider $X>0$ and 
$$p(s,a,X)=\mathds{1}_{\{a-s>X\}},$$
then there are no steady states of the linear Equation \eqref{eql2t}, besides the zero solution. Indeed, from Equation \eqref{Idcompact} we deduce that the discharging flux $N$ should satisfy
\begin{equation*}
    N(a)=e^{-a}\int_X^\infty N(u)\,du,
\end{equation*}
whose unique non-negative solution in $L^1(0,\infty)$ is $N\equiv0$.
\end{rmk}

\section{Steady states}
\label{stationary2t}
Consider $n^*=n^*(s,a)$ with support in the set $\{s\le a\}$. We are interested in the stationary solutions of the non-linear System \eqref{eqmain2t} given by
\begin{equation}
\label{eqest2t}
\left\{
\begin{matrix*}[l]
\partial_s n+\partial_a n+p(s,a,X)n=0&a>s>0,\vspace{0.15cm}\\
n(s=0,a)=N(a)\coloneqq\int_0^\infty p(a,u,X)n(a,u)\,du&a>0,\vspace{0.15cm}\\
X=\int_0^\infty N(a)\,da,\vspace{0.15cm}\\
\iint n(s,a)\,da\,ds=1,\quad n(s,a)\ge0.
\end{matrix*}
\right.
\end{equation}

We define $N_X$ as the respective conditional activity in terms of $X$. In order to have a steady state of the non-linear Problem \eqref{eqmain2t}, we must find $X>0$ such that
\begin{equation}
\label{eqphi}
    X=\Phi(X)\coloneqq\int_0^\infty N_X(a)\,da.
\end{equation}
In the general case this equation has always a solution since the right-hand side is uniformly bounded thanks to estimate \eqref{boundX} and $N_X(a)$ depends continuously on $X$. By using the properties of the operator $\mathcal{T}_X$, we prove that under the weak interconnections regime the non-linear System \eqref{eqmain2t} has a unique steady state.

\begin{thm}
\label{uniquesteady}
Assume \eqref{mass2t} and that $p$ smooth satisfies Assumption \eqref{boundp2t}. Then for $\|\partial_X p\|_\infty$ small enough, System \eqref{eqmain2t} has a unique steady state $(n^*,N^*,X^*)$.
\end{thm}

\begin{proof}
The goal is to prove that $\Phi$ is a contraction in order to obtain a unique fixed point. In order to estimate $\partial_X N$ we make use of the implicit function theorem. By differentiating Equation \eqref{normal2t} we get
\begin{equation}
\label{normal2tx}
    \iint \partial_X(N_X)e^{-\int_0^s p(s',a+s',X)\,ds'}\,da\,ds=\iint N_Xe^{-\int_0^s p(s',a+s',X)\,ds'}\left(\int_0^s \partial_X p(s',u'+s',X)\,ds'\right)\,da\,ds.
\end{equation}
Furthermore, if we differentiate with respect to $X$ the Equation \eqref{Idcompact}, we get
\begin{equation*}
\begin{split}
    \partial_X N_X(a)&=\int_0^\infty \left(-\partial_a e^{-\int_0^a p(s',u+s',X)\,ds'}\right)\partial_X N_X(u)\,du\\
    &\quad+\int_0^\infty \left(-\partial_a\partial_X e^{-\int_0^a p(s',u+s',X)\,ds'}\right)N_X(u)\,du,
\end{split}
\end{equation*}
i.e. $\partial_X N_X(a)$ satisfies the equation
\begin{equation}
\begin{split}
    (I-\mathcal{T}_X)[\partial_X N](a)&=\int_0^\infty p(a,u+a,X)e^{-\int_0^a p(s',u+s',X)\,ds'}N_X(u)\,du\\
    &\quad-\int_0^\infty p(a,u+a,X)\left(\int_0^a \partial_X p(s',u+s',X)\,ds'\right)e^{-\int_0^a p(s',u+s',X)\,ds'}N_X(u)\,du.
\end{split}
\end{equation}
By using the implicit function theorem and the Condition \eqref{normal2tx} we can define an inverse of $I-\mathcal{T}_X$ which depends continuously on $X$. Observe that $\|(I-\mathcal{T}_X)^{-1}\|$ is uniformly bounded on $X$ in the operator norm, since $X$ is uniformly bounded. Thus, for the function $\Phi$ we get
\begin{equation*}
    \begin{split}
        |\Phi'(X)|&=\left|\int \partial_X N_X(a)\,da\right|\\
        &\le \|(I-\mathcal{T}_X)^{-1}\|\,\|\partial_X p\|_\infty\left(\iint (1+p_\infty a) e^{-\int_0^a p(s',u+s',X)\,ds}N_X(u)da\,du\right)\\
        & \le C\|\partial_X p\|_\infty \left(\iint (1+p_\infty a) e^{-p_0 a} N_X(u)\,da\,du\right)\\
        & \le C\|\partial_X p\|_\infty p_\infty \left(\int (1+p_\infty a)e^{-p_0 a}\,da\right),
    \end{split}
\end{equation*}
so that for $\|\partial_X p\|_\infty$ small enough $\Phi$ is a contraction and we conclude the result.
\end{proof}

\section{Convergence to equilibrium}
\label{convergence2t}
After studying the linear case, we are now ready to prove convergence to the steady steady under the weak interconnection regime, i.e. $\|\partial_X p\|_\infty$ small enough, by a perturbation argument.
\begin{thm}[Convergence to equilibrium]
	\label{conveq2tnl}
	Assume that $n_0\in L^1(\mathcal{D})$ satisfies Assumption \eqref{mass2t} and that $p$ Lipschitz satisfies Assumption \eqref{boundp2t}. For $\|\partial_X p\|_\infty$ small enough, let $(n^*,N^*,X^*)$ be the corresponding stationary state of System \eqref{eqmain2t}. Then there exist $C,\lambda>0$ such that the solution $n$ of System \eqref{eqmain2t} satisfies
	$$\|n(t)-n^*\|_{L^1_{s,a}}\le Ce^{-\lambda t}\|n_0-n^*\|_{L^1_{s,a}},\:\forall t\ge0.$$
	Moreover $\|N(t)-N^*\|_{L^1}$ and $|X(t)-X^*|$ converge exponentially to $0$ when $t\to\infty$.
\end{thm}

\begin{proof}
Observe that $n$ satisfies the evolution equation
	$$\partial_t n=\mathcal{L}_X [n]\coloneqq-\partial_s n-\partial_a n-p(s,a,X(t))n+\delta_{\{s=0\}}(s,a)\int_0^\infty p(a,u,X(t))n(t,a,u)\,du,$$
	where $\delta_{\{s=0\}}(s,a)$ is the measure along the line $\{(0,a)\colon a\ge0\}$. 
	We can rewrite the evolution equation as
	\begin{equation}
	\label{nt=Ln}
	\partial_t n=\mathcal{L}_{X^*}[n]+(\mathcal{L}_{X}[n]-\mathcal{L}_{X^*}[n])=\mathcal{L}_{X^*}[n]+h.
	\end{equation}
	with $h(t,s,a)$ given by
	\begin{equation}
	h=\big(p(s,a,X^*)-p(s,a,X(t)\big)n(t,s,a)+\delta_{\{s=0\}}(s,a)\int_0^\infty\big(p(a,u,X(t))-p(a,u,X^*)\big)n(t,a,u)\,du.
	\end{equation}
	Let $P_t\colon L^1(\mathcal{D})\to L^1(\mathcal{D})$ be the linear semi-group associated to operator $\mathcal{L}_{X^*}$. As in the proof of Lemma \ref{linearaux2t}, $P_t$ is extended to space $(\mathcal{M}(\mathcal{D}),\|\cdot\|_{M^1})$ in order to be able to evaluate at the measure $h$. Since $P_t n^*=n^*$ for all $t\ge0$, we get that $n$ satisfies
	\begin{equation}
	\label{n-n*2}
	n-n^*=P_t (n_0-n^*)+\int_0^t P_{t-\tau}h(\tau,s,a)\,d\tau,
	\end{equation}
	so we need find an estimate for the function $h$. Observe that we have the following inequalities:
	\begin{equation*}
		\begin{matrix}
		\|h(t)\|_{L^1_{s,a}}\le 2\|\partial_X p\|_\infty|X(t)-X^*|\vspace{0.15cm},\\
		|X(t)-X^*|\le \|N(t)-N^*\|_1,\vspace{0.15cm}\\
		\|N(t)-N^*\|_1\le\|\tfrac{\partial p}{\partial X}\|_\infty|X(t)-X^*|+p_\infty\|n(t)-n^*\|_{L^1_{s,a}},
		\end{matrix}
	\end{equation*}
	and since $\|\partial_X p\|_\infty<1$ we get
	\begin{equation*}
		\begin{matrix}
		\|h(t)\|_{L^1_{s,a}}\le \frac{2p_\infty\|\partial_X p\|_\infty}{1-\|\partial_X p\|_\infty}\|n(t)-n^*\|_{L^1_{s,a}}\vspace{0.15cm},\\
		|X(t)-X^*|\le \frac{p_\infty}{1-\|\partial_X p\|_\infty}\|n(t)-n^*\|_{L^1_{s,a}}\vspace{0.15cm},\\
		\|N(t)-N^*\|_1\le p_\infty\left(\frac{\|\partial_X p\|_\infty}{1-\|\partial_X p\|_\infty}+1\right)\|n(t)-n^*\|_{L^1_{s,a}},
		\end{matrix}
	\end{equation*}
	thus by taking norm in Equality \eqref{n-n*2} and applying Doeblin's Theorem we obtain
	\begin{equation*}
	\begin{split}
		\|n(t)-n_*\|_{L^1_{s,a}}&\le \|P_t(n_0-n_*)\|_{L^1_{s,a}}+\int_0^t\|P_{t-\tau}h(\tau)\|_{L^1_{s,a}}\,d\tau\\
		&\le \frac{e^{-\lambda t}}{1-\alpha}\|n_0-n_*\|_{L^1_{s,a}}+\frac{1}{1-\alpha}\int_0^te^{-\lambda(t-\tau)}\|h(\tau)\|_{L^1_{s,a}}\,d\tau\\
		&\le\frac{e^{-\lambda t}}{1-\alpha}\|n_0-n_*\|_{L^1_{s,a}}+C\int_0^t e^{-\lambda(t-\tau)}\|n(\tau)-n^*\|_{L^1_{s,a}}\,d\tau,
	\end{split}
	\end{equation*}
	with $C\coloneqq\frac{1}{1-\alpha}\frac{2p_\infty\|\partial_X p\|_\infty}{1-\|\partial_X p\|_\infty}$. By using Gronwall's inequality with respect to the function $e^{\lambda t}\|n(t)-n^*\|_{L^1_{s,a}}$ we conclude
	\begin{equation*}
	    \|n(t)-n^*\|_{L^1_{s,a}}\le \frac{e^{-(\lambda-C)t}}{1-\alpha} \|n_0-n^*\|_{L^1_{s,a}},
	\end{equation*}
	so that for $\|\partial_X p\|_\infty$ small enough we have $C<\lambda$ and we deduce the exponential convergence of $n(t,\cdot,\cdot),N(t,\cdot)$ and $X(t)$ when $t\to\infty$.
\end{proof}

\section{Numerical simulations}
\label{numerical2t}
In order to illustrate the theoretical  long time results and other possible behaviors of System \eqref{eqmain2t}, we present numerical simulations for different firing rates and initial data. The numerical illustrations  below are obtained by solving the equation~\eqref{eqmain2t} with a classical first-order upwind scheme.

We focus in displaying the discharging flux $N(t,a)$ and the total activity $X(t)$ since these two elements determine the general behavior of system \eqref{eqmain2t}.

\subsection{Example 1: Convergence to equilibrium}
For our first example, we choose as initial data $n_0(s,a)=e^{-a}$ and the firing rate is given by
$$p=\mathds{1}_{\{s>X\}}+\mathds{1}_{\{s-a>X\}},$$
which corresponds to an inhibitory case since $p$ is decreasing with respect to $X$. Moreover, this particular form of $p$ is decomposed as the sum of two simple threshold functions with the first one depending only on the first elapsed time and the second one depending on the difference between the  last two discharges.

\begin{figure}[ht!]
	\centering
	\begin{subfigure}{0.48\textwidth}
		\includegraphics[width=\textwidth]{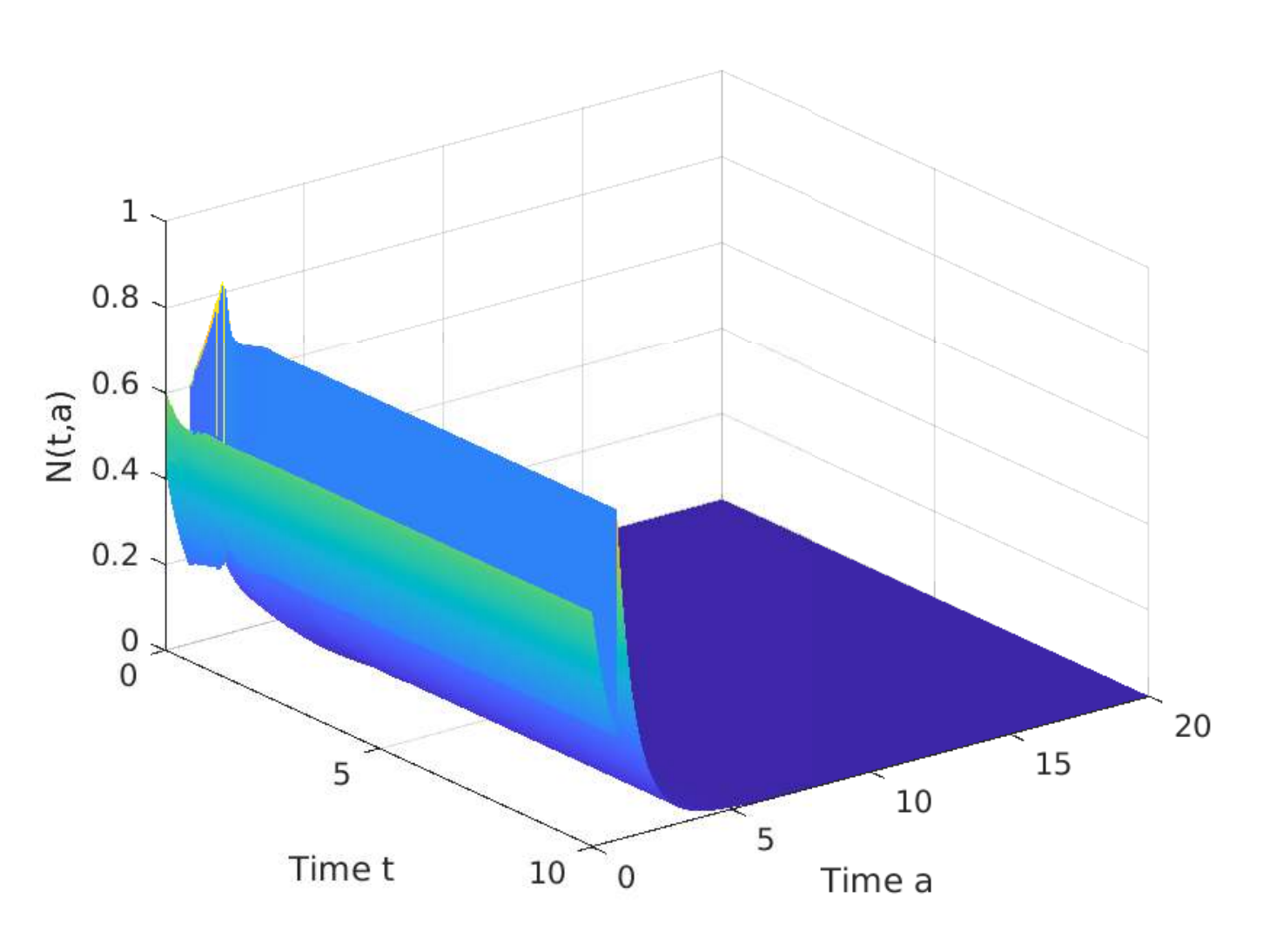}
		\caption{Activity $N(t,x)$.}
	\end{subfigure}  
	\begin{subfigure}{0.48\textwidth}
		\includegraphics[width=\textwidth]{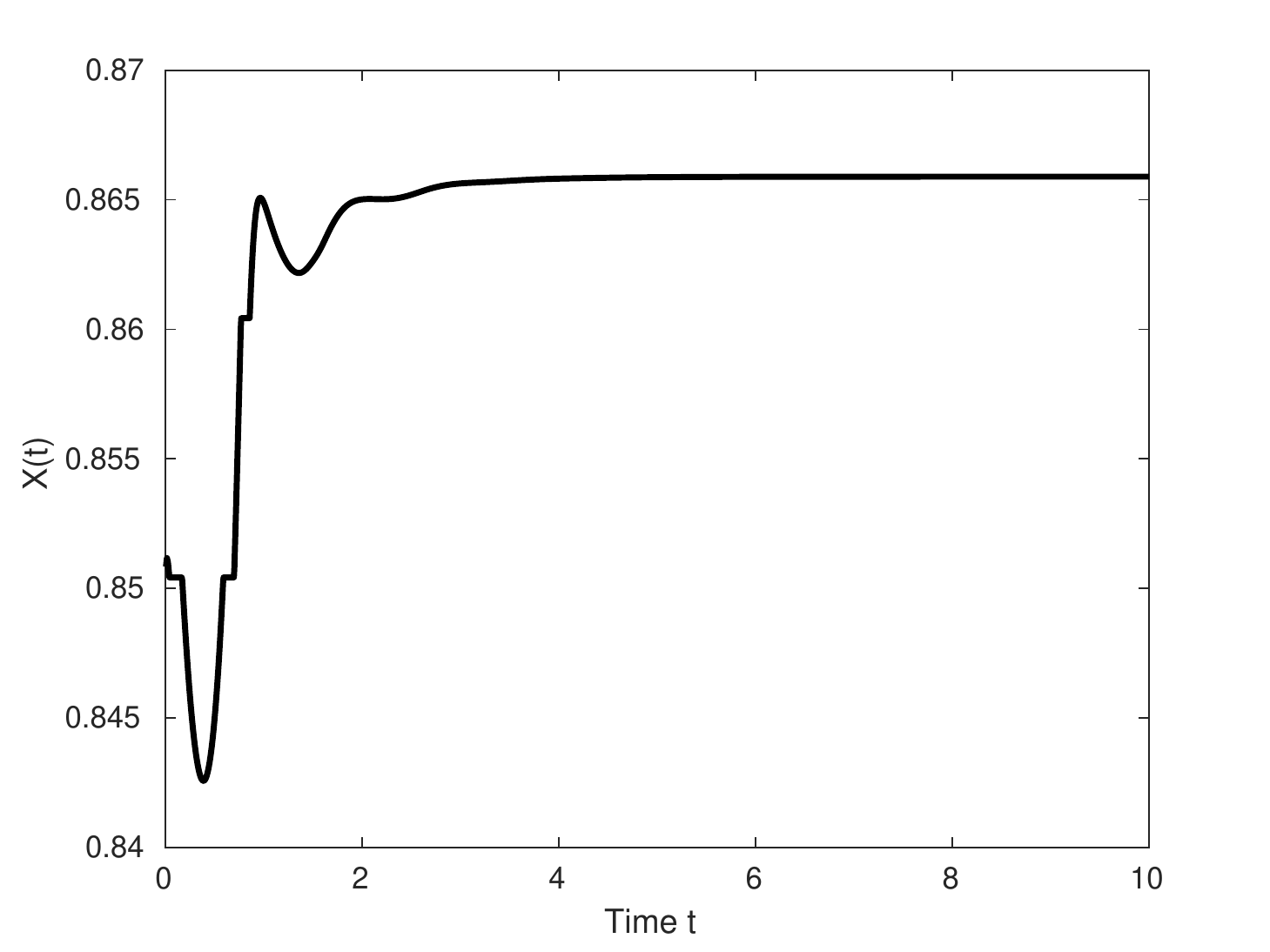}
		\caption{Total Activity $X(t)$.}
	\end{subfigure}
	 \caption{\textbf{Example 1.} Case $n_0(s,a)=e^{-a}$ and $p=\mathds{1}_{\{s>X\}}+\mathds{1}_{\{s-a>X\}}$.}
	\label{Example1-2t}
\end{figure}

In this case the solution simply converges to a steady state, as we see in Figure \ref{Example1-2t} for the discharging flux $N$ and the total activity $X$. From Equation \eqref{Idcompact} we note that the discharging flux at equilibrium $N^*$ has a jump discontinuity at $X^*$, which is consistent with the numerical solution. This convergence is compatible with Theorem \ref{conveq2tnl}.

\subsection{Example 2: Jump discontinuities}
We now consider the initial data $n_0(s,a)=2\cdot\mathds{1}_{\{2>a>s+1\}}$ and the firing rate is given 
$$p=\mathds{1}_{\{s>e^{-X}\}}+\mathds{1}_{\{s-a>e^{-X}\}},$$
which corresponds to an excitatory case since $p$ is increasing with respect to $X$.
\begin{figure}[ht!]
	\centering
	\begin{subfigure}{0.48\textwidth}
		\includegraphics[width=\textwidth]{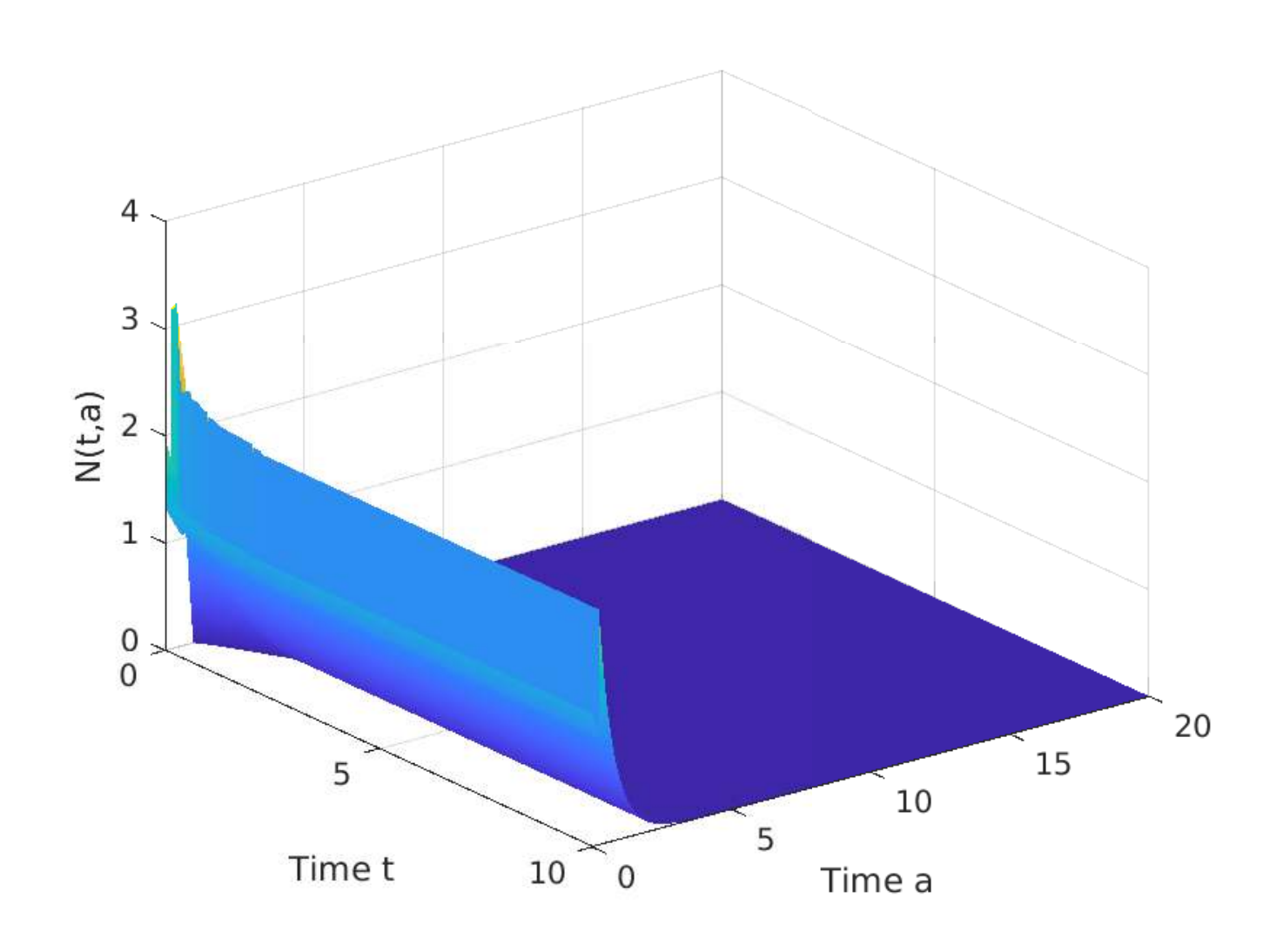}
		\caption{Activity $N(t,x)$.}
	\end{subfigure}  
	\begin{subfigure}{0.48\textwidth}
		\includegraphics[width=\textwidth]{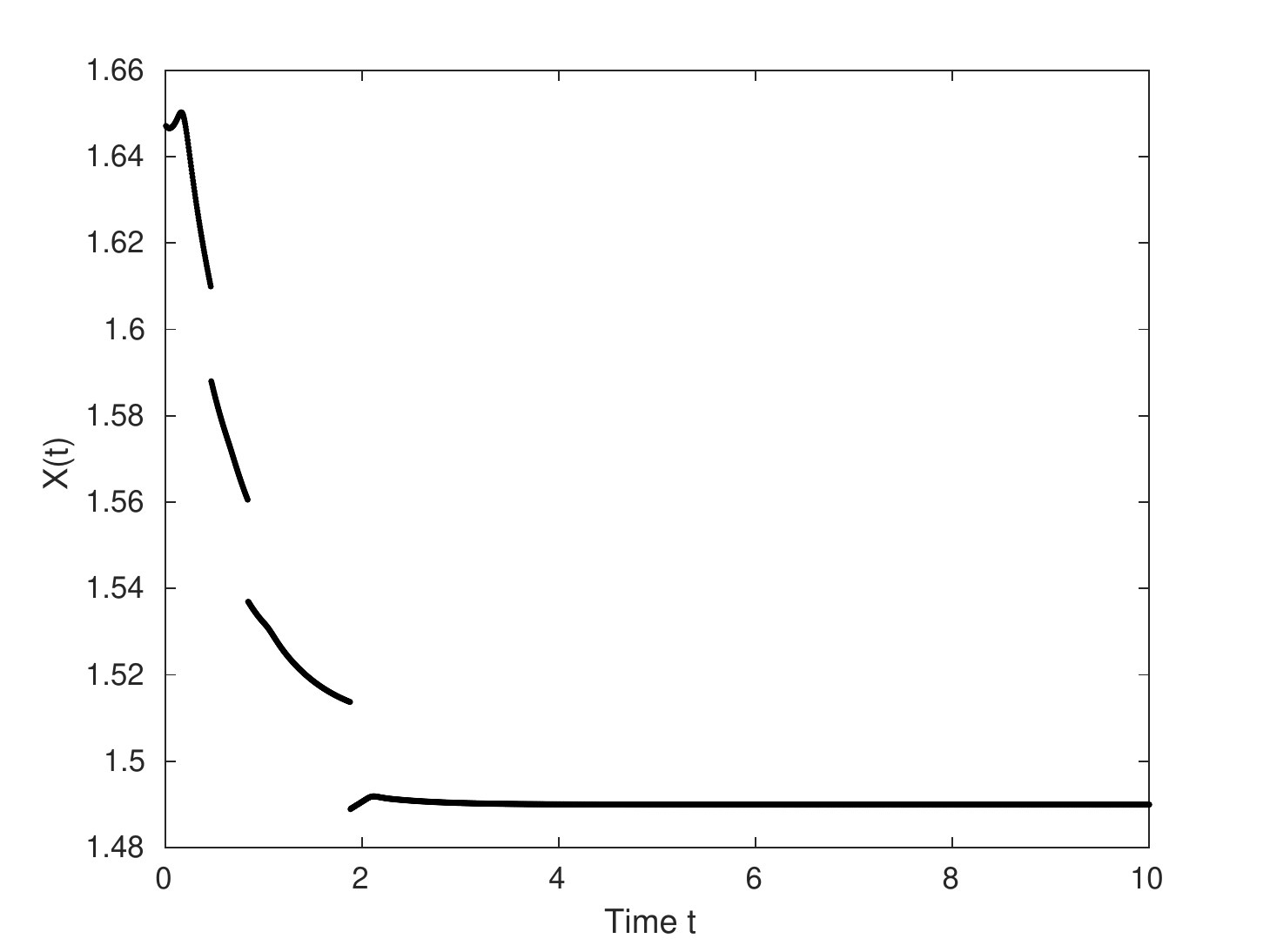}
		\caption{Total Activity $X(t)$.}
	\end{subfigure}
	\caption{\textbf{Example 2.} Case $n_0(s,a)=2\cdot\mathds{1}_{\{2>a>s+1\}}$ and $p=\mathds{1}_{\{s>e^{-X}\}}+\mathds{1}_{\{s-a>e^{-X}\}}$.}
	\label{Example2-2t}
\end{figure}

Like the previous example the solution converges to the steady state, but the total activity $X$ shows three jump discontinuities as we see in Figure \ref{Example2-2t}. The multiple jump discontinuities are consequence of the contribution of the term depending on the difference between the two elapsed times. Furthermore, solutions convergent to the steady state that present a single jump discontinuity were already observed in Caceres et al. \cite{torres2021elapsed} for the classical elapsed time model. The phenomenon of multiple jumps discontinuities in Figure \ref{Example2-2t} is an extension for the case of Equation \eqref{eqmain2t}.

\subsection{Example 3: Periodic solutions and stabilization}
Next, we choose initial data $n_0(s,a)=\frac{1}{2}e^{-(a-1)}\mathds{1}_{\{a>\max(s,1)\}}$ and the firing rate is given by
$$p=\varphi(X)\mathds{1}_{\{s>1\}},\qquad\varphi(u)=\frac{10u^2}{u^2+1}+0.5,$$
which corresponds to an excitatory case since $\varphi'(u)>0$. Since $p$ does not depend on $a$, we take advantage by solving the classical elapsed time Equation \eqref{eqm} after integrating with respect to $a$, as we remarked in the introduction.

\begin{figure}[ht!]
	\centering
	\begin{subfigure}{0.48\textwidth}
		\includegraphics[width=\textwidth]{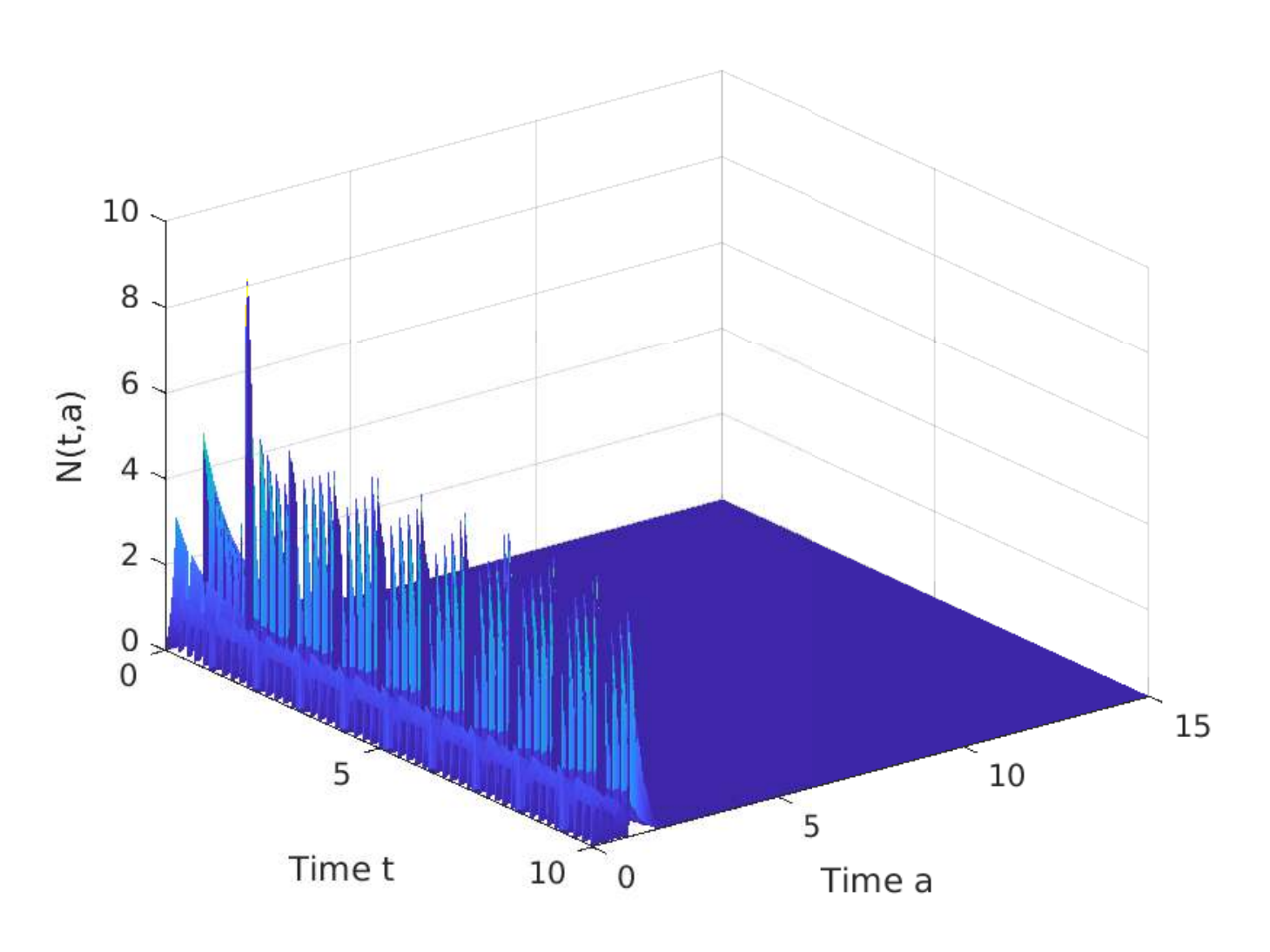}
		\caption{Activity $N(t,x)$.}
	\end{subfigure}  
	\begin{subfigure}{0.48\textwidth}
		\includegraphics[width=\textwidth]{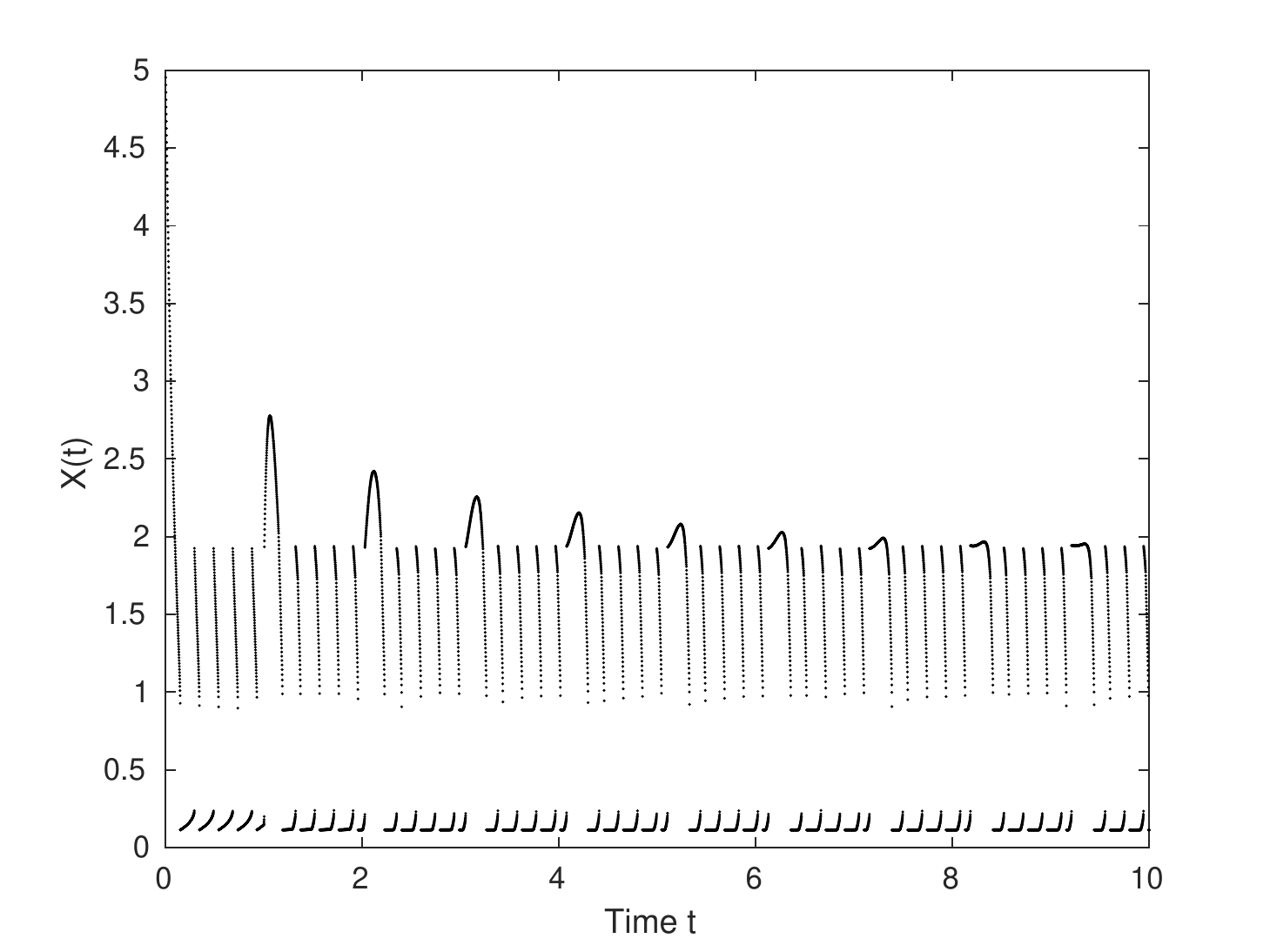}
		\caption{Total Activity $X(t)$.}
	\end{subfigure}
	\caption{\textbf{Example 3.1.} Case $n_0(s,a)=\frac{1}{2}e^{-(a-1)}\mathds{1}_{\{a>\max(s,1)\}}$ and $p=\varphi(X)\mathds{1}_{\{s>1\}}$.}
	\label{Example3a-2t}
\end{figure}
For these data, both the discharging flux $N$ and the total activity $X$ are asymptotic to a periodic pattern as we see in Figure~\ref{Example3a-2t}. Similar examples on periodic solutions were found in Caceres et al.~\cite{torres2021elapsed} in the classical elapsed time model for the same type of firing rates.

However, when we incorporate the effects of the difference between the elapsed times the periodic regime changes. For the same initial data and
$$p=\varphi(X)\mathds{1}_{\{s>1\}}+\mathds{1}_{\{s-a>X\}},$$
we observe in Figure \ref{Example3b-2t} that, with the term depending on the difference between the two elapsed times, the solution  of System~\eqref{eqmain2t}  converges to the steady state.
\begin{figure}[ht!]
	\centering
	\begin{subfigure}{0.48\textwidth}
		\includegraphics[width=\textwidth]{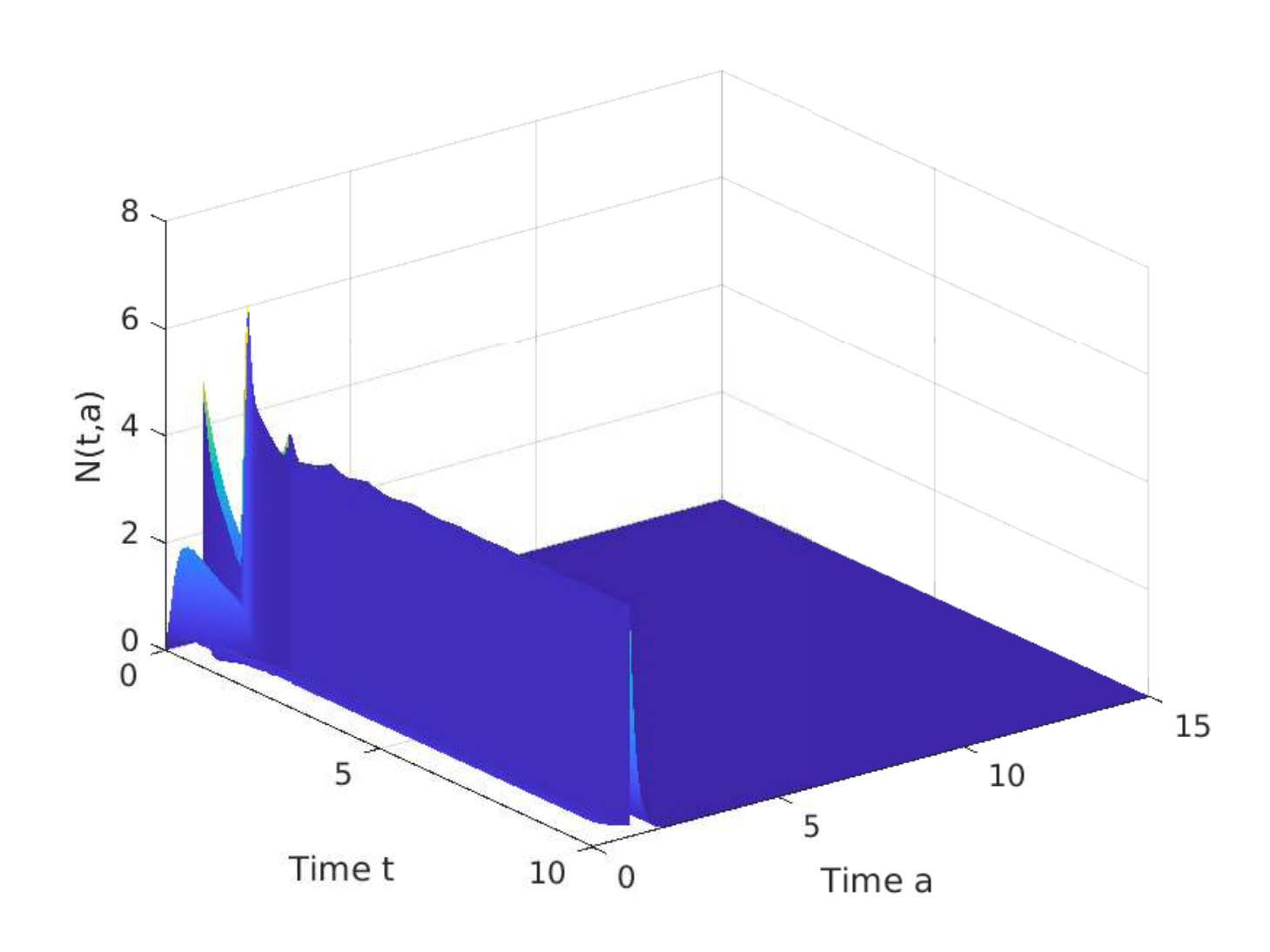}
		\caption{Activity $N(t,x)$.}
	\end{subfigure}  
	\begin{subfigure}{0.48\textwidth}
		\includegraphics[width=\textwidth]{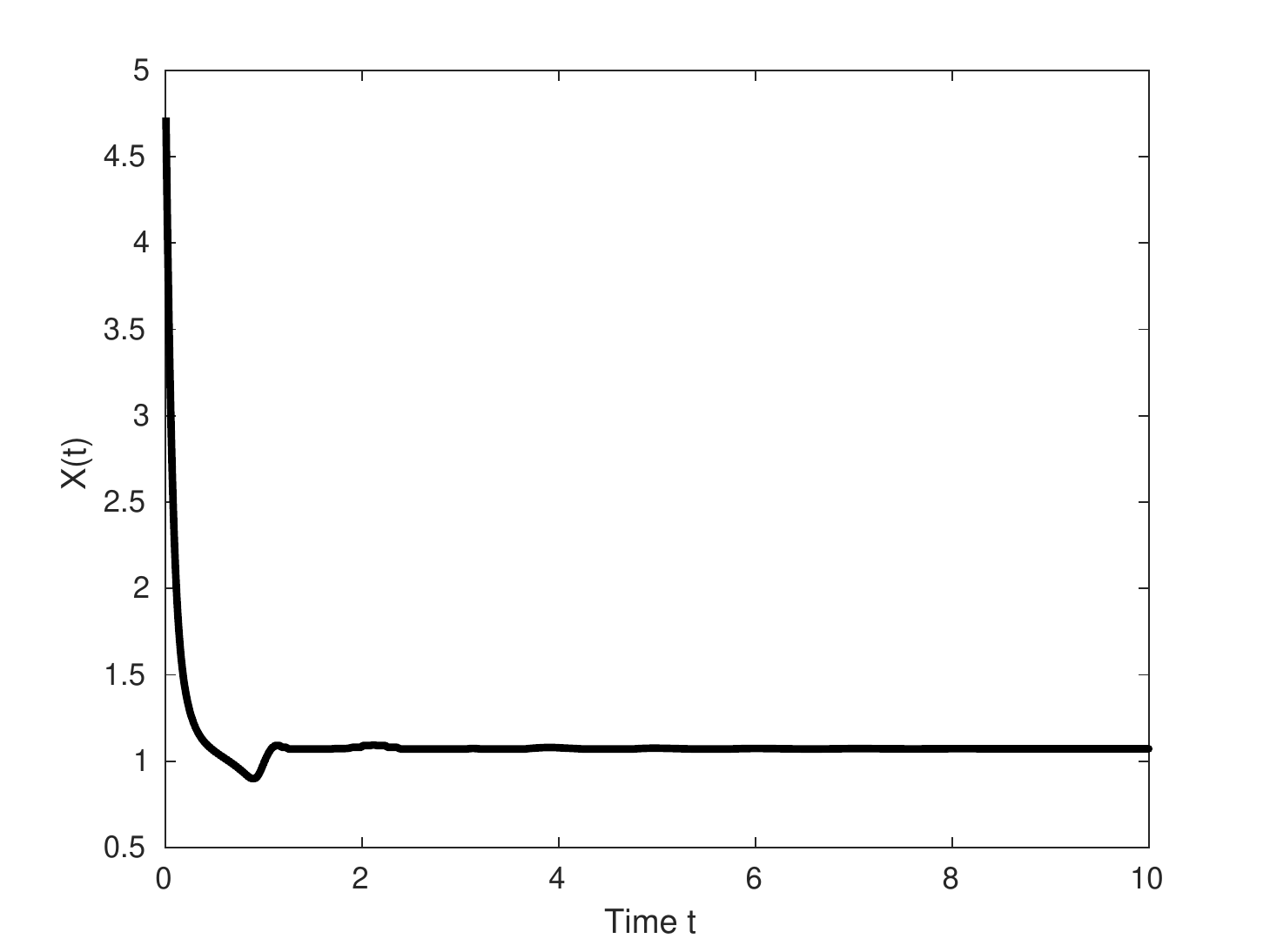}
		\caption{Total Activity $X(t)$.}
	\end{subfigure}
	\caption{\textbf{Example 3.2.} Case $n_0(s,a)=\frac{1}{2}\,e^{-(a-1)}\mathds{1}_{\{a>\max(s,1)\}}$ and $p=\varphi(X)\mathds{1}_{\{s>1\}}+\mathds{1}_{\{s-a>X\}}$.}
	\label{Example3b-2t}
\end{figure}

\section{Perspectives}
By means of Doeblin's theory applied to a more complex problem that the classical elapsed time model, we managed to understand the dynamics of System \eqref{eqmain2t} for weak non-linearities by adapting the ideas of Cañizo et al.~\cite{canizo2019asymptotic}. However, aspects such as well-posedness and the asymptotic behavior for strong interconnections are still an open problem as in the classical elapsed time model.

Concerning the strongly inhibitory case, it remains pending to prove uniqueness of the steady state. Whilst in the classical elapsed time model this problem is reduced to a simple equation, for the model with two elapsed times we have to prove uniqueness for the integral Equation \eqref{eqphi}. Moreover, we conjecture in the general case that the speed of convergence to a steady state must be exponential like it is expected for the classical elapsed equation.

With respect to the existence of periodic solutions, we still have to find or construct a non-trivial example relying on dynamics for two elapsed times. The only examples we have found so far are adaptations of solutions of the classical elapsed time equation that were obtained in Caceres et al.~\cite{torres2021elapsed} and these types of solutions presents jump discontinuities, making them difficult to analyze. Furthermore, it remains as an open problem to find continuous periodic solutions as in the classical elapsed time model. 

\section*{Acknowledgements}
NT has received funding from the European Union's Horizon 2020 research and innovation program under the Marie Sklodowska-Curie grant agreement No 754362. BP has received funding from the European Research Council (ERC) under the European Union's Horizon 2020 research and innovation programme grant agreement No 740623. DS has received support from ANR ChaMaNe No: ANR-19-CE40-0024.
\\[-35pt]
\begin{center}
    \includegraphics[scale=0.05]{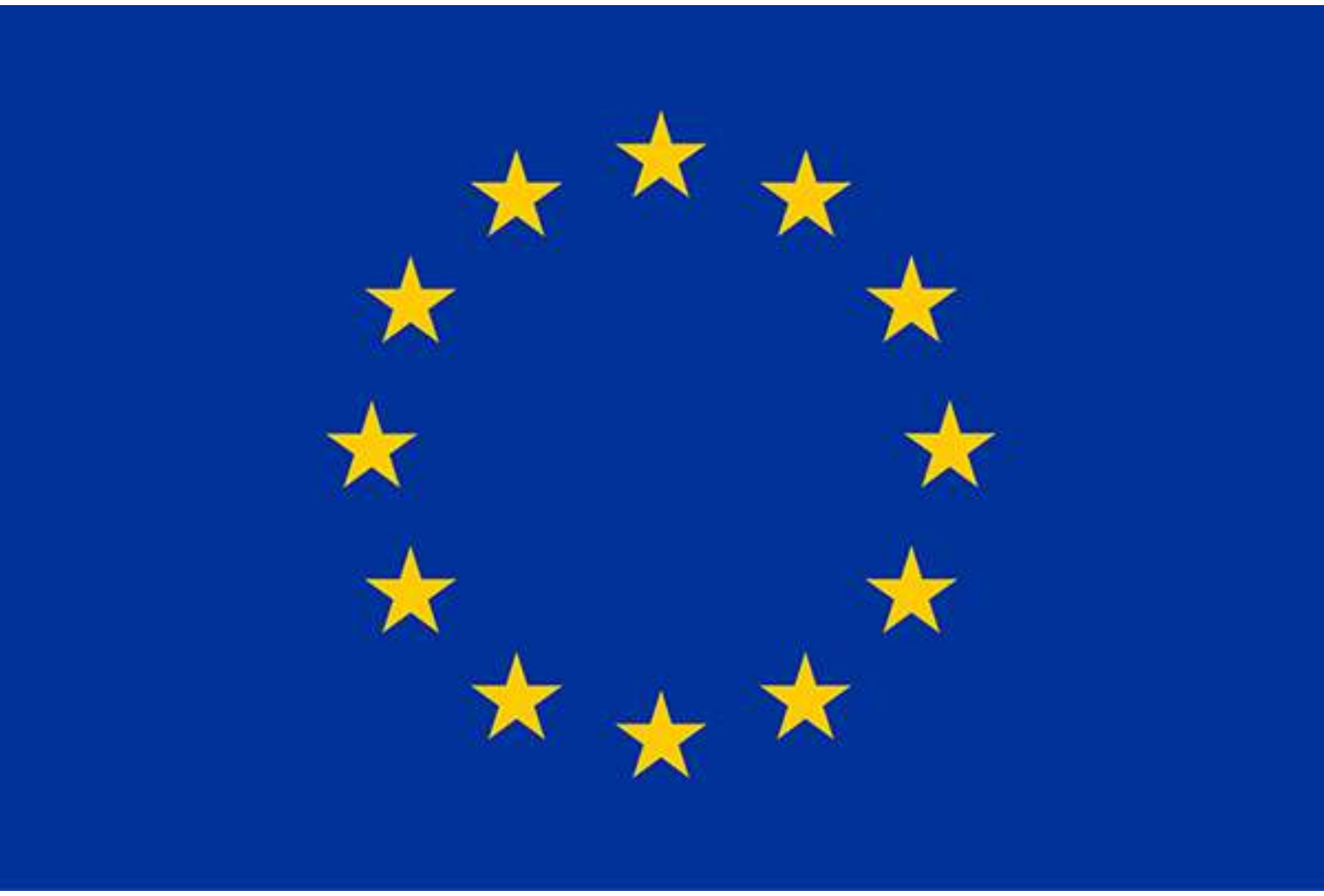}
\end{center}


\bibliography{Thesis.bib}
\bibliographystyle{plain}
\end{document}